\documentclass[11pt]{article}

\usepackage[utf8]{inputenc}
\usepackage[dvipsnames]{xcolor}

\usepackage[pdfencoding=auto, psdextra, colorlinks,citecolor=blue!90!black]{hyperref}       %
\usepackage{booktabs}       %
\usepackage{amsfonts}       %
\usepackage{nicefrac}       %
\usepackage{mathrsfs}
\usepackage{graphicx}
\usepackage{amsmath,amssymb,amsthm}
\usepackage[sans]{dsfont}
\usepackage[round]{natbib}
\usepackage{cleveref}
\usepackage{pifont}
\usepackage{algorithm,algorithmicx,algpseudocode}
\usepackage{thmtools}
\usepackage{thm-restate}
\def\O#1{\text{\ding{\the\numexpr#1+171}}}
\declaretheorem[name=Theorem]{Theorem}
\declaretheorem[name=Problem,sibling=Theorem]{Problem}
\declaretheorem[name=Lemma,sibling=Theorem]{Lemma}
\declaretheorem[name=Proposition,sibling=Theorem]{Proposition}
\declaretheorem[name=Fact,sibling=Theorem]{Fact}
\declaretheorem[name=Corollary,sibling=Theorem]{Corollary}
\declaretheorem[name=Example,sibling=Theorem]{Example}
\declaretheorem[name=Remark,sibling=Theorem]{Remark}
\declaretheorem[name=Definition,sibling=Theorem]{Definition}

\declaretheorem[name=Assumption,sibling=Theorem]{Assumption}
\usepackage{bm}

\def\dist{\textnormal{dist}}
\def\sgn{\textnormal{sgn}}

\def\conv{\textnormal{Conv}}

\def\spn{\textnormal{span}}

\title{Testing Stationarity Concepts for ReLU Networks: \\ Hardness, Regularity, and Robust Algorithms}

\date{\today}

\author{Lai~Tian \thanks{Department of Systems Engineering and Engineering Management, The Chinese University of Hong Kong, Sha Tin, N.T., Hong Kong SAR. E-mail: \href{mailto:tianlai@se.cuhk.edu.hk}{\tt tianlai@se.cuhk.edu.hk}.} \and 
Anthony~Man-Cho~So \thanks{Department of Systems Engineering and Engineering Management, The Chinese University of Hong Kong, Sha Tin, N.T., Hong Kong SAR. E-mail: \href{mailto:manchoso@se.cuhk.edu.hk}{\tt manchoso@se.cuhk.edu.hk}.}
}
\usepackage{geometry}
\usepackage{mathtools}
\usepackage{enumitem}
\usepackage{framed}

\def\lB{\left\{}
\def\rB{\right\}}
\def\rk{\mathop{\textnormal{rank}}}
\def\St{\mathop{\textnormal{St}}}
\def\argmin{\mathop{\textnormal{argmin}}}

\def\eargmax{\mathop{\textnormal{$\epsilon$--argmax}}}
\def\epsilon{\varepsilon}

\def\diag{\mathop{\textnormal{Diag}}}

\def\geq{\geqslant}
\def\leq{\leqslant}

\renewcommand{\cite}{\citep}

\begin{document}
\maketitle

\begin{abstract}
	We study the computational problem of the stationarity test for the empirical loss of neural networks with ReLU activation functions. Our contributions are:
	\begin{enumerate}
		\item Hardness: We show that checking a certain first-order approximate stationarity concept for a piecewise linear function is  co-NP-hard. This implies that testing a certain stationarity concept for a modern nonsmooth neural network is in general computationally intractable. As a corollary, we prove that testing so-called first-order minimality for functions in abs-normal form is co-NP-complete, which was conjectured by \citet[SIAM J. Optim., vol.~29, p284]{griewank2019relaxing}.		
	\item Regularity: We establish a necessary and sufficient condition for the validity of an equality-type subdifferential chain rule in terms of Clarke, Fr\'echet, and limiting subdifferentials of the empirical loss of two-layer ReLU networks. %
	This new condition is simple and efficiently checkable.
		\item Robust algorithms: We introduce an algorithmic scheme to test near-approximate stationarity in terms of both Clarke and Fr\'echet subdifferentials. Our scheme makes no false positive or false negative error when the tested point is sufficiently close to a stationary one and a certain qualification is satisfied. This is the first practical and robust stationarity test approach for two-layer ReLU networks.
	\end{enumerate}
\end{abstract}

\newpage
\section{Introduction}
The theoretical analysis of ReLU neural network training is challenging from the optimization perspective, though the empirical performance of various ``gradient''-based algorithms is surprisingly good. A key difficulty comes from the entanglement of nonconvexity and nonsmoothness in the objective function of the empirical loss, which causes not only the notion of gradient from classical analysis meaningless, but also the subdifferential set from convex analysis vacuous. Consequently, the study of such a nonconvex nondifferentiable function requires the use of tools from variational analysis \cite{rockafellar2009variational}. 

For a continuously differentiable function $f:\mathbb{R}^d\rightarrow \mathbb{R}$, a point $\bm{x} \in \mathbb{R}^d$ is called stationary (or critical) if $\nabla f(\bm{x}) = \bm{0}$. However, the situation is much more complicated when $f$ is nondifferentiable at $\bm{x}$. 
Indeed, there are many different stationarity concepts (see \Cref{def:stationarity}) for nonsmooth functions \cite{li2020understanding,cui2021modern}.
For general Lipschitz functions, recently, under the oracle complexity framework of \citet{nemirovskij1983problem}, substantial progress has been made on the design of provable algorithms for finding approximately stationary (in the sense of perturbed) points \cite{zhang2020complexity,tian2022finite,davis2021gradient,lin2022gradient,metel2022perturbed,kong2022cost} and also on establishing the hardness of computing such approximate stationary points \cite{kornowski2021oracle,tian2022no,kornowski2022complexity,jordan2022complexity}.

As a complement to these developments, in this paper, we consider the complexity of and robust algorithms for checking whether a given neural network is an (approximately) stationary one with respect to the empirical loss. This is a task already considered by \citet{yun2018efficiently}. We emphasize that ``checking'' and ``finding'' are two very different computational problems. While the co-NP-hardness of checking the local optimality of a given point in smooth nonconvex programming was shown by \citet{murty1987some} in 1987, the complexity of ``finding'' a local minimizer was an open question proposed by \citet{pardalos1992open} since 1992, and is recently settled by \citet{ahmadi2022complexity}.

Given a neural network with smooth elemental components, testing the (approximate) stationarity of a point is simply an application of the classic gradient chain rule. In a modern computational environment, this is usually done by using Algorithmic Differentiation (AD) \cite{griewank2008evaluating} software, e.g., PyTorch and TensorFlow. A natural question that arises is whether testing  the stationarity for a piecewise smooth function (e.g., empirical loss of a ReLU network) is as easy as testing for a smooth one. Surprisingly, we show (in \Cref{thm:hard-general}) that such testing is, in general, computationally intractable.

The difficulty here is due to the failure of an exact (equality-type) subdifferential chain rule. For a general locally Lipschitz function, the calculus rules are only known to hold in the form of set inclusions rather than equalities, except in several special cases (see \Cref{fact:rule}). This prevents one from computing the subdifferential set of the empirical loss with that of elemental components. Thus, to facilitate the tractability of stationarity testing, it is of interest to find out a condition, under which an equality-type chain rule holds,  and the subdifferential set of the empirical loss can be characterized. 
By contrast, given a first-order oracle providing the whole generalized subdifferential set at the queried point in the oracle framework \cite{kornowski2021oracle,tian2022no,kornowski2022complexity,jordan2022complexity}, the stationarity testing task reduces to a simple linear program, which can be solved by interior-point methods in polynomial time. However, in practice, even computing an element in the generalized subdifferential for a nonsmooth function can be highly non-trivial \cite{burke2002approximating,nesterov2005lexicographic,ma2010computation,khan2013evaluating}. Therefore, a condition for the validity of the exact chain rule could be useful for subgradient computation and stationarity testing and analysis.

The most closely related work to ours is the one by \citet{yun2018efficiently}. They considered a two-layer ReLU network and introduced a theoretical algorithm to sequentially check Clarke stationarity (see \Cref{def:stationarity}), Fr\'echet stationarity, and a certain second-order optimality condition. For Fr\'echet stationarity testing, they proposed to verify the nonnegativity of a directional derivative in every possible direction, for which a trivial test in the worst case requires checking exponentially many inequalities. By exploiting polyhedral geometry, they showed that it suffices to check only extreme rays, which can be done in polynomial time. A limitation of the work \cite{yun2018efficiently} (see also the discussion in \cite[Section 5]{yun2018efficiently}) is that the algorithm therein can only perform \emph{exact} stationarity testing (see \Cref{sec:test-exact}). That is to say if the objective function is $x\mapsto |x|$, then the algorithm in \cite{yun2018efficiently} will certify stationarity if and only if $x=0$. However, as pointed out by \citet[Section 5]{yun2018efficiently}, in practice, such an exact nondifferentiable point is almost impossible to reach. Therefore, it is desirable to have a robust stationarity testing algorithm that works for points sufficiently close to a stationary one. In other words, we are interested in testing so-called near-approximate stationarity (see \Cref{def:nas}). We mention that, without exploiting structures in the nonsmooth objective function, such robust testing is impossible in general \cite[Theorem 2.7]{tian2022no}.

\subsection{Our Results and Techniques}
\paragraph{Hardness.} Our first main result shows that checking certain first-order approximate stationarity concept for an unconstrained piecewise differentiable function is co-NP-hard (see \Cref{thm:hard-general}). This implies that testing a certain stationarity concept for a shallow modern convolutional neural network is co-NP-hard (see \Cref{coro:nnhard}). Our reduction is from the 3-satisfiability (3SAT) to a stationarity testing problem. As a corollary, we prove that testing so-called first-order minimality (FOM) for functions in abs-normal form is co-NP-complete (see \Cref{coro:abs-norm-hard}) and give an affirmative answer to a conjecture of \citet[SIAM J. Optim., vol.~29, p284]{griewank2019relaxing}.		

Our other results concern the empirical loss of a two-layer ReLU network, which was also studied by \citet{yun2018efficiently}.
Given the training data $\{(\bm{x}_i, y_i)\}_{i=1}^N \subseteq \mathbb{R}^d\times \mathbb{R}$ with the $\bm{x}=(\tilde{\bm{x}},1)$ parametrization, we first make the following blanket assumptions.
\begin{Assumption}[Blanket assumptions]\label{assu:loss}\
	The loss function $\ell:\mathbb{R}\times \mathbb{R}\rightarrow\mathbb{R}$ is smooth and has locally Lipschitz gradient. For simplicity of notation, we write $\ell_i(\cdot)$ for $\ell(\cdot, y_i)$.  For any $i \in [N]$, we assume $\bm{x}_i \neq \bm{0}$, which is superfluous for the $\bm{x}=(\tilde{\bm{x}},1)$ parametrization.
\end{Assumption}
The empirical loss of a two-layer ReLU neural network with $H$ hidden nodes can be written as
\[
 L(u_1,\bm{w}_1,\dots,u_H,\bm{w}_H) \coloneqq \sum_{i=1}^N \ell_i \left( \sum_{k=1}^H u_k\cdot \max\left\{ \bm{w}_k^\top \bm{x}_i, 0 \right\} \right).
\]

\paragraph{Regularity.} By naïvely abusing the convex subdifferential chain rule for $L$, we consider the following  ``generalized subdifferential'' of the empirical loss $L$ as
\[
\tilde{G}\coloneqq \sum_{i=1}^N \rho_i\cdot\prod_{k=1}^H\lB \max\left\{ \bm{w}_k^\top \bm{x}_i, 0 \right\} \rB\times
\left\{ \begin{array}{rcl}
         \lB u_k\cdot\bm{x}_i\cdot\mathbf{1}_{\bm{w}_k^\top \bm{x}_i > 0} \rB & \mbox{if}
         & \bm{w}_k^\top \bm{x}_i \neq 0, \\
         u_k\cdot\bm{x}_i\cdot [0,1]  & \mbox{if} & \bm{w}_k^\top \bm{x}_i = 0,
                \end{array}\right.
\]
with $\rho_i\coloneqq \ell_i' \left( \sum_{k=1}^H u_k\cdot \max\left\{ \bm{w}_k^\top \bm{x}_i, 0 \right\} \right),\forall i \in [N]$. This ``generalized subdifferential'' is popular in practical computation and theoretical analysis. For example, see \cite[Equation (9)]{wang2019learning}, \cite[Section 3.1]{arora2019fine}, and \cite[Equations (5) and (6)]{safran2022effective}. However, as $L$ is nonconvex and nonsmooth, we can only assert a fuzzy chain rule (see \cite[Section 2.3]{clarke1990optimization}) for the Clarke subdifferential $\partial_C L$ of $L$, which is a set inclusion $\partial_C L(u_1,\bm{w}_1,\dots,u_H,\bm{w}_H) \subseteq \tilde{G}$ rather than an equation.

Our second main result is a necessary and sufficient condition for the validity of a series of equality-type subdifferential chain rules for the empirical loss of this shallow ReLU network. We show that, under this regularity condition, exact chain rules hold for three commonly used generalized subdifferentials, i.e., Clarke (see \Cref{def:subd} and \Cref{thm:reluMain-CR}), limiting (see \Cref{def:subd-l} and \Cref{thm:reluMain-CR-limiting}), and Fr\'echet (see \Cref{def:subd-f} and \Cref{thm:reluMain-CR-frechet}). 
It is notable that while sufficient conditions for the equality-type calculus rules are rather rich in the literature (see \cite[Chapter 10]{rockafellar2009variational}), a necessary condition is rarely seen, let alone an efficiently computable, necessary and sufficient condition in our \Cref{thm:reluMain-CR}.

\paragraph{Robust algorithms.} Our third main result is an algorithmic scheme to test the so-called near-approximate stationarity (see \Cref{def:nas}) in terms of both Clarke and Fr\'echet subdifferentials. We show that, for an approximate stationary point $\bm{x}^*$, any point that is sufficiently close to $\bm{x}^*$ can be certified (with \Cref{alg:robust}) as near-approximate stationary.  Our technique is a new rounding scheme (see \Cref{alg:rounding}) motivated by the notion of active manifold identification \cite{lewis2002active,lemarechal2000u} in the literature. This new rounding scheme is capable of identifying the activation pattern of the target stationary point and finding a nearby point with the same pattern. One notable application of such a near-approximate stationarity test is to obtain a termination criterion for algorithms that only have asymptotic convergence results. 
For example, every limiting point of the sequence generated by the stochastic subgradient method has been shown to be Clarke stationary (see \Cref{def:stationarity}) by \citet[Corollary 5.11]{davis2020stochastic}, %
but it is still unclear when to terminate the algorithm, and how to certify  the obtained point is at least close to some Clarke stationary point, as the norm of any vector in the subdifferential is almost surely lower bounded away from zero during the entire trajectory  (consider running the subgradient method on $x\mapsto |x|$). 	

\paragraph{Notation.} %
Scalars, vectors and matrices are denoted by lowercase letters, boldface lower case letters, and boldface uppercase letters, respectively.
The notation used in this paper is mostly standard:  $\mathbb{B}_\epsilon(\bm{x})\coloneqq\{\bm{v}:\|\bm{v} - \bm{x}\|\leq \epsilon\}$ (we may write $\mathbb{B}_\epsilon^d(\bm{x})$ to emphasize the dimension);  $\dist(\bm{x},S)\coloneqq\inf_{\bm{v}\in S} \|\bm{v}-\bm{x}\|$ for a closed set $S$, which is defined as $+\infty$ if the set $S=\emptyset$; $\conv(S)$ denotes the convex hull of the set $S$; the vector $\bm{e}_i$ denotes the $i$-th column of identity matrix $\bm{I}$;  $\mathbb{R}_+\coloneqq \{x \in \mathbb{R}: x \geq 0\}$; $\pi_i$ denotes the project to the $i$-th argument operator; i.e., $\pi_i \left(\prod_{j=1}^n S_j\right) \coloneqq S_i$ for sets $\{S_i\}_{i=1}^n$; the extended-real $\overline{\mathbb{R}}$ is defined as $\mathbb{R}\cup\{-\infty,+\infty\}$; the addition of two sets is always understood in the sense of Minkowski; $\overline{\mathbb{Z}}\coloneqq \mathbb{Z}\cup \{-\infty, \infty\}$; $[m]\coloneqq \{1,\dots,m\}$ for any integer $m \geq 1$.

\paragraph{Organization.} We introduce the background on generalized differentiation theory and formal definitions of stationarity concepts in Section~\ref{sec:prel}. 
Then, in Section~\ref{sec:hardness}, we present our main hardness results. The necessary and sufficient condition of the validity of chain rule in terms of various subdifferential constructions is presented in \Cref{sec:regularity}. We discuss the robust algorithms to test near-approximate stationarity concepts in \Cref{sec:test}. All proofs are deferred to the Appendices.

\section{Preliminaries}\label{sec:prel}

The following construction of subdifferential by \citet[Theorem 2.5.1]{clarke1990optimization} is classic.
\begin{Definition}[Clarke subdifferential]\label{def:subd} Given a point $\bm{x}$, the Clarke subdifferential of a locally Lipschitz function $f$ at $\bm{x}$ is defined by
	\[
	\partial_C f(\bm{x}) := \conv\big\{\bm{s}:\exists \bm{x}^\prime\! \rightarrow\! \bm{x}, \nabla f(\bm{x}^\prime) \textnormal{ exists}, \nabla f(\bm{x}^\prime)\!\rightarrow\! \bm{s}\big\}.
	\]
\end{Definition}
For a locally Lipschitz function, the Clarke subdifferential is always nonempty, convex, and compact \cite[Proposition 2.1.2(a)]{clarke1990optimization}. The following set generated by a directional derivative $f'$ is known as the Fr\'echet subdifferential of $f$ \cite[Exercise 8.4]{rockafellar2009variational}.
\begin{Definition}[Fr\'echet subdifferential]\label{def:subd-f}
	Given a point $\bm{x}$, the Fr\'echet subdifferential of a locally Lipschitz and directional differentiable function $f$ at $\bm{x}$ is defined by
	\[
	\widehat{\partial} f (\bm{x}) := \big\{\bm{s}: \bm{s}^\top \bm{d} \leq f'(\bm{x};\bm{d}) \text{ for all } \bm{d}\big\}.
	\]
\end{Definition}
 The set-valued  mapping $\widehat\partial f$ of Fr\'echet subdifferential of $f$ is not outer semicontinuous (see \cite[Definition 5.4]{rockafellar2009variational}), which means that given $\bm{x}_\nu \rightarrow \bm{x}, \bm{g}_\nu \rightarrow \bm{g}$ with $\bm{g}_\nu \in \widehat{\partial} f(\bm{x}_\nu)$, we cannot assert $\bm{g} \in \widehat{\partial} f(\bm{x})$. The following limiting subdifferential (or the Mordukhovich subdifferential) \cite[Definition 8.3(b)]{rockafellar2009variational} is more robust for analysis.
 \begin{Definition}[Limiting subdifferential]\label{def:subd-l}
	Given a point $\bm{x}$, the limiting subdifferential of a locally Lipschitz and directional differentiable function $f$ at $\bm{x}$ is defined by
	\[
	\partial f (\bm{x}) := \limsup_{\bm{x}^\prime \rightarrow \bm{x}} \widehat{\partial} f (\bm{x}^\prime),
	\]
	where the outer limit is taken in the sense of Kuratowski (see, e.g., \cite[p152, Equation 5(1)]{rockafellar2009variational}).
\end{Definition}
In the following  result, we record a generalized Fermat's rule for optimality conditions and the relationship among the aforementioned three subdifferentials.
\begin{Fact}[{\citet[Theorem 8.6, 8.49, 10.1]{rockafellar2009variational}}]\label{fact:relation-fermat}
Given a locally Lipschitz function $f:\mathbb{R}^d\rightarrow\mathbb{R}$ and a point $\bm{x} \in \mathbb{R}^d$, then we have
$
 \widehat{\partial} f(\bm{x}) \subseteq \partial f(\bm{x}) \subseteq \partial_C f(\bm{x}).
$
If the point $\bm{x}$ is a local minimizer of the function $f$, then it holds that  $\bm{0} \in \widehat{\partial} f(\bm{x})$.
\end{Fact}
We are now ready to state the definitions of various stationarity concepts.
\begin{Definition}[Stationarity concepts]\label{def:stationarity}
	Given a locally Lipschitz function $f:\mathbb{R}^d\rightarrow\mathbb{R}$, we say that the point $\bm{x}\in\mathbb{R}^d$ is an
	\begin{itemize}
		\item $\epsilon$-Clarke stationary point if $\dist\big(\bm{0}, \partial_C f(\bm{x})\big) \leq \epsilon$;
		\item $\epsilon$-Fr\'echet stationary point if $\dist\big(\bm{0}, \widehat{\partial} f(\bm{x})\big) \leq \epsilon$;
		\item $\epsilon$-limiting stationary point if $\dist\big(\bm{0}, \partial f(\bm{x})\big) \leq \epsilon$.
	\end{itemize}
\end{Definition}

The following Clarke regularity for locally Lipschitz and directional differentiable functions is a classic notion related to the validity of various subdifferential calculus rules; see \cite[Definition 2.3.4]{clarke1990optimization} and \cite[Corollary 8.11]{rockafellar2009variational}. 
\begin{Definition}[Clarke regularity]\label{def:clarke-regular}
For a locally Lipschitz directional differentiable function $f:\mathbb{R}^d\rightarrow\mathbb{R}$ and a point $\bm{x}$, one has $f$ is Clarke regular at $\bm{x}$ if $\partial_C f(\bm{x}) = \widehat{\partial} f(\bm{x})$.
\end{Definition}

We record some basic equality-type calculus rules for Clarke subdifferential as follows; see {\cite[Proposition 2.3.3, Theorem 2.3.10]{clarke1990optimization}, and \cite[Proposition 2.5]{rockafellar1985extensions}}. We refer the reader to \cite[Chapter 10]{rockafellar2009variational} for similar calculus rules for Fr\'echet and limiting subdifferentials.
\begin{Fact}[Calculus rules]\label{fact:rule} Let $f:\mathbb{R}^d\rightarrow \mathbb{R},g:\mathbb{R}^d\rightarrow \mathbb{R}$ be two locally Lipschitz functions.
\begin{itemize}
	\item If $f$ is strictly differentiable at $\bm{x}$, then $\partial_C (f+g)(\bm{x}) = \nabla f(\bm{x}) + \partial_C g(\bm{x})$;
	\item If $h(\bm{x},\bm{y}) = f(\bm{x}) + g(\bm{y})$, then $\partial_C h(\bm{x},\bm{y}) = \partial_C f(\bm{x}) \times \partial_C g(\bm{y})$;
	\item Given a strictly differentiable mapping $G:\mathbb{R}^n\rightarrow\mathbb{R}^d$ and a point $\bm{y}\in\mathbb{R}^n$, if the function $f$ (or $-f$) is Clarke regular at $G(\bm{y})$, then $f\circ G$ (or $-f\circ G$) is Clarke regular at $\bm{y}$ and $\partial_C [f\circ G](\bm{y}) = (JG(\bm{y}))^\top \partial_C f(G(\bm{y}))$, where $JG$ is the Jacobian of mapping $G$. The equality also holds when $JG$ is surjective.
\end{itemize}
\end{Fact}
\begin{Remark}
The sum rule is a special case of the chain rule, which does not hold for Lipschitz functions trivially. For example, consider $\partial_C[|\cdot| - |\cdot|](0) = \{0\} \subsetneq \partial_C[|\cdot|](0) + (-\partial_C[|\cdot|](0))=[-2,2]$.
The empirical loss of a ReLU network is in general not Clarke regular. To see this, let $f(x,y)=\max\{x,0\}-\max\{y,0\}$. It is easy to verify neither $f$ nor $-f$ is Clarke regular. Another remark here is on the notion of partial subdifferentiation; see \cite[Corollary 10.11]{rockafellar2009variational} and \cite[p48]{clarke1990optimization}. 
 In general, we cannot say much about the relationship between $\partial f(\bm{x},\bm{y})$ and $\partial_{\bm{x}} f(\bm{x},\bm{y}) \times \partial_{\bm{y}} f(\bm{x},\bm{y})$ (see \cite[Example 2.5.2]{clarke1990optimization}), except the following inclusion \cite[Proposition 2.3.16]{clarke1990optimization}:
$
\partial_{\bm{x}} f(\bm{x},\bm{y}) \times \partial_{\bm{y}} f(\bm{x},\bm{y}) \subseteq \pi_1 \partial f(\bm{x},\bm{y}) \times \pi_2 \partial f(\bm{x},\bm{y}).
$
\end{Remark}

\section{Hardness of Stationarity Testing}\label{sec:hardness}

For smooth nonconvex programming, co-NP-hardness has been shown for local optimality testing \cite[Theorem 2]{murty1987some} and second-order sufficient condition testing \cite[Theorem 4]{murty1987some}. However, in the nonsmooth case,  we show that checking a first-order necessary condition approximately in terms of certain subdifferential is already co-NP-hard. 

\begin{Theorem}[Testing of piecewise linear functions]\label{thm:hard-general} 
Given a $3\sqrt{d}$-Lipschitz piecewise linear function $f:\mathbb{R}^d\rightarrow \mathbb{R}$ in the form of max–min representation\footnote{Any piecewise linear function $f:\mathbb{R}^d\rightarrow\mathbb{R}$ can be written using a max-min representation as $
f(\bm{x}) = \max_{1\leq i \leq l} \min_{j \in M_i} \bm{a}_j^\top \bm{x} + b_j,
$ where $M_i\subseteq [m]$ is a finite index set; see \cite[Proposition 2.2.2]{scholtes2012introduction}. The input data are $d\in\mathbb{N}, m \in \mathbb{N}, l \in \mathbb{N}, \{(\bm{a}_j, b_j)\}_{j=1}^m,$ and  $\{M_i\}_{i=1}^l.$} with integer data. For any $\eta \in (d, +\infty]\cap \overline{\mathbb{Z}}$,
	checking whether the point $\bm{0}\in\mathbb{Z}^d$ satisfying $\dist\big(\bm{0}, \widehat\partial f(\bm{0})\big) \leq \nicefrac{1}{\sqrt{\eta}}$ is co-NP-hard, and checking whether $\bm{0} \in \widehat\partial f(\bm{0})$ is strongly co-NP-hard.
\end{Theorem}

We compare \Cref{thm:hard-general} with the classic hardness result of \citet{murty1987some}. In \cite{murty1987some}, checking the local optimality of a simply constrained indefinite quadratic problem \cite[Problem 1]{murty1987some}  and of an unconstraint quartic polynomial objective \cite[Problem 11]{murty1987some} are both co-NP-complete. However, these hardness results are inapplicable for checking first-order necessary conditions. In fact, for any hard construction $f:\mathbb{R}^n \rightarrow \overline{\mathbb{R}}$ in \cite{murty1987some} and a given point $\bm{x} \in \mathbb{Q}^n$, testing $\bm{0}\in\widehat{\partial}f(\bm{x})$ can be done in polynomial time with respect to the input size. In \Cref{thm:hard-general}, we show that for a class of simple unconstrained piecewise differentiable functions, even an approximate test of the first-order necessary condition $\bm{0}\in\widehat{\partial}f(\bm{x})$ for a certain point $\bm{x}$ is already computationally intractable. 

Nonsmooth functions in real-world applications usually contain structures that can be exploited in theoretical analysis and algorithmic design. A subclass of piecewise differentiable functions, termed $C^d_{\textnormal{abs}}$ or functions representable in abs-normal form, and defined as the composition of smooth functions and the absolute value function, is introduced by \citet{griewank2013stable}; see \Cref{sec:abs-norm-form} for a brief introduction and \cite[Definition 2.1]{griewank2019relaxing} for details.
An important corollary of our hard construction concerns the complexity of checking an optimality condition for functions in $C^d_{\textnormal{abs}}$. The following result gives an affirmative answer to a conjecture of \citet[p284]{griewank2019relaxing}:

\begin{Corollary}[Testing of abs-normal form]\label{coro:abs-norm-hard}
	Testing first order minimality (FOM) for a piecewise differentiable  function given in the abs-normal form is co-NP-complete. 
\end{Corollary}

Now, we report another notable corollary about the complexity of testing a certain stationarity concept for the empirical loss of a modern convolutional neural network.
\begin{Corollary}[Testing of loss of nonsmooth networks]\label{coro:nnhard}
Let $f:\mathbb{R}^d\rightarrow\mathbb{R}$ be the empirical loss function of a shallow neural network with ReLU activation function, max-pooling operator, and convolution operator. Suppose the width of the first layer is $m$.
Then, for any $\eta \in (m, +\infty]\cap \overline{\mathbb{Z}}$,
testing the $\nicefrac{1}{\sqrt{\eta}}$-Fr\'echet stationarity $\dist\big(\bm{0}, \widehat\partial f(\bm{\theta})\big) \leq \nicefrac{1}{\sqrt{\eta}}$ for a certain $\bm{\theta}\in\mathbb{Q}^d$ is co-NP-hard, and testing $\bm{0}\in \widehat\partial f(\bm{\theta})$ for $\bm{\theta}$ is strongly co-NP-hard.
\end{Corollary}

\Cref{coro:nnhard} shows a computational tractability separation for the stationarity test between smooth and nonsmooth networks. In the smooth setting, given the gradient of every component function, we can compute the gradient norm of the loss function by iteratively applying chain rule. But in the nonsmooth case, while the subdifferential of every elemental function can be computed easily, the validity of the subdifferential chain rule like those in \Cref{fact:rule} is not justified, which turns out to cause a serious computational hurdle in stationarity test (strong co-NP-hardness). %

\section{Regularity Conditions}\label{sec:regularity}

In this section, we study the regularity conditions for the validity of the equality-type chain rule  in terms of Clarke, Fr\'echet, and limiting subdifferentials of the empirical loss of two-layer ReLU networks.  
\subsection{Setup}
For simplicity of reference, we introduce the following notation, which will be used in various  subdifferential constructions of the empirical loss $L$.

\begin{Definition} Let the parameters $\{(u_k, \bm{w}_k)\}_{k=1}^H$ be given. We define the following shorthands:
\begin{enumerate}[label=\textnormal{(\alph*)}]
	\item We write constants $\rho_i \coloneqq \ell_i'\left( \sum_{k=1}^H u_k\cdot \max\left\{ \bm{w}_k^\top \bm{x}_i, 0 \right\} \right) $ for any $i \in [N]$.
	\item For any $k \in [H]$ and $\bm{w}_k \in \mathbb{R}^d$, we define the following two indices sets:
\[
\begin{aligned}
\mathcal{I}_k^+(\bm{w}_k) &\coloneqq \left\{ i: \bm{w}_k^\top \bm{x}_i = 0, u_k\cdot \rho_i \geq 0, i \in [N] \right\}, \\
\mathcal{I}_k^-(\bm{w}_k) &\coloneqq \left\{ i: \bm{w}_k^\top \bm{x}_i = 0, u_k\cdot \rho_i < 0,  i \in [N] \right\}.	
\end{aligned}
\]
We may write $\mathcal{I}_k^+$ and $\mathcal{I}_k^-$ when the reference point $\bm{w}_k$ is clear from the context.
\item For any $k \in [H]$, we define the following nonempty convex compact set $G^C_k \subseteq \mathbb{R}^d$ related to the Clarke subdifferential:
\[
 G^C_k\coloneqq \sum_{i \in [N]\backslash (\mathcal{I}_k^+\cup\mathcal{I}_k^-)} u_k\rho_i\cdot \mathbf{1}_{\bm{w}_k^\top\bm{x}_i > 0} \cdot\bm{x}_i+ \sum_{j\in  \mathcal{I}_k^+\cup\mathcal{I}_k^-} u_k\rho_j\cdot \bm{x}_j\cdot [0,1].
\]
\item For any $k \in [H]$, we define the following nonempty compact set $G^L_k\subseteq \mathbb{R}^d$ related to the limiting subdifferential:
\[
\begin{aligned}
	G^L_k&\coloneqq  \sum_{i \in [N]\backslash (\mathcal{I}_k^+\cup\mathcal{I}_k^-)} u_k\rho_i\cdot \mathbf{1}_{\bm{w}_k^\top\bm{x}_i > 0} \cdot\bm{x}_i \\ 
	&\qquad+ \sum_{j\in  \mathcal{I}_k^+} u_k\rho_i\bm{x}_j\cdot[0,1]
	+
 \lB \sum_{j\in  \mathcal{I}_k^-} u_k\rho_i\cdot \mathbf{1}_{\bm{d}^\top\bm{x}_j > 0} \cdot\bm{x}_j:\exists\bm{d} \in \mathbb{R}^d, \min_{t \in   \mathcal{I}_k^-}\left| \bm{x}_t^\top \bm{d} \right| > 0
 \rB.
\end{aligned}
\]
\item For any $k \in [H]$, we define the following convex compact set $G^F_k\subseteq \mathbb{R}^d$ related to the Fr\'echet subdifferential:
\[
G^F_k\coloneqq \sum_{i \in [N]\backslash (\mathcal{I}_k^+\cup\mathcal{I}_k^-)} u_k\rho_i\cdot \mathbf{1}_{\bm{w}_k^\top\bm{x}_i > 0} \cdot\bm{x}_i + \sum_{j\in  \mathcal{I}_k^+} u_k\rho_j\bm{x}_j\cdot[0,1]+
 \left\{ \begin{array}{rcl}
          \emptyset & \mbox{if}
         & \left| \mathcal{I}_k^- \right| > 0, \\
        \bm{0} & \mbox{if} & \left| \mathcal{I}_k^- \right| = 0.
                \end{array}\right.
\]
\item If an equation holds for all the three subdifferentials, i.e., Clarke/limiting/Fr\'echet subdifferentials ($\partial_Cf/\partial f/\widehat{\partial}f$), we will write the equation simply with $\partial_\triangleleft f$ and also $G_k^\triangleleft$ (for $G_k^C/G_k^L/G_k^F$). For example, if the equation $\partial_\triangleleft f_k(\bm{w}_k) = G_k^\triangleleft$  holds , then we get $\partial_C f_k(\bm{w}_k) = G_k^C, \partial f_k(\bm{w}_k) = G_k^L$, and $\widehat{\partial} f_k(\bm{w}_k) = G_k^F$.
\end{enumerate}
\end{Definition}

\subsection{Main Results}
\begin{Theorem}[Clarke chain rule]\label{thm:reluMain-CR} Under \Cref{assu:loss},
we claim that the exact Clarke subdifferential chain rule holds for $L$ at a given point $(u_1, \bm{w}_1, \dots, u_H, \bm{w}_H)$, that is 
\[
\partial_C L(u_1,\bm{w}_1,\dots,u_H,\bm{w}_H) = \prod_{k=1}^H \left\{  \sum_{i=1}^N \rho_i \cdot \max\left\{ \bm{w}_k^\top \bm{x}_i, 0 \right\}  \right\} \times G^C_k,
\]
if and only if
 the data points $\{\bm{x}_i\}_{i=1}^N$ satisfy the following Span Qualification (SQ):
\begin{framed}
	\[\label{eq:sq}
	\bigcup_{1\leq k \leq H}\spn \left(\lB \bm{x}_i \rB_{i\in\mathcal{I}_k^+} \right) \cap 
	\spn \left(\lB \bm{x}_j \rB_{j\in\mathcal{I}_k^-} \right) = \{\bm{0}\}. \tag{SQ}\vspace{-2.5mm}
	\] 
\end{framed}
\end{Theorem}

\begin{Remark}
Note that for any $k\in[H]$, the indices sets $\mathcal{I}_k^-$ and $\mathcal{I}_k^+$ can be computed in $O(Nd)$. Then, checking SQ is no harder than checking the Linear Independence Constraint Qualification (LICQ) in nonlinear programming and can be done with, e.g., Zassenhaus algorithm.
\end{Remark}

\begin{Theorem}[Limiting chain rule]\label{thm:reluMain-CR-limiting} Under \Cref{assu:loss},
we claim that the exact limiting subdifferential chain rule holds for $L$ at a given point $(u_1, \bm{w}_1, \dots, u_H, \bm{w}_H)$, that is 
\[
\partial L(u_1,\bm{w}_1,\dots,u_H,\bm{w}_H) = \prod_{k=1}^H \left\{   \sum_{i=1}^N \rho_i \cdot \max\left\{ \bm{w}_k^\top \bm{x}_i, 0 \right\}  \right\} \times G^L_k,
\]
if and only if
 the data points $\{\bm{x}_i\}_{i=1}^N$ satisfy SQ.
\end{Theorem}

For Fr\'echet subdifferential, the situation is different as the default chain rule is the reverse set inclusion $\widehat{\partial} L(u_1,\bm{w}_1,\dots,u_H,\bm{w}_H) \supseteq \prod_{k=1}^H \left\{ \sum_{i=1}^N \rho_i \cdot \max\left\{ \bm{w}_k^\top \bm{x}_i, 0 \right\}  \right\} \times G^F_k$; see \cite[Corollary 10.9, Theorem 10.49]{rockafellar2009variational}. If $\widehat{\partial} L(u_1,\bm{w}_1,\dots,u_H,\bm{w}_H) = \emptyset$, we have  the exact chain rule trivially, as $G^F_k$ can only be the empty set. Therefore, the interesting case is when the Fr\'echet subdifferential is nonempty. 
\begin{Theorem}[Fr\'echet chain rule]\label{thm:reluMain-CR-frechet} Under \Cref{assu:loss}, for any given point such that the subdifferential $\widehat{\partial} L(u_1,\bm{w}_1,\dots,u_H,\bm{w}_H) \neq \emptyset$, we have the following exact chain rule for the empirical loss $L$ 
\[
\widehat{\partial} L(u_1,\bm{w}_1,\dots,u_H,\bm{w}_H) = \prod_{k=1}^H \left\{ \sum_{i=1}^N \rho_i \cdot \max\left\{ \bm{w}_k^\top \bm{x}_i, 0 \right\}  \right\} \times G^F_k,
\]
if and only if
 the data points $\{\bm{x}_i\}_{i=1}^N$ satisfy SQ.
 \end{Theorem}

\subsection{Discussion}\label{sec:discussion}
There are several existing regularity conditions related to the validity of exact chain rule of the empirical loss. We briefly introduce them here and defer the details to the \Cref{def:detailed-regularity} in \Cref{sec:apd-prf-discussion}.
\begin{Definition}[Regularities]\label{def:other-regularity} We consider the following regularity conditions:
\begin{itemize}
	\item General position data: \cite[Section 2.2]{montufar2014number}, \cite[Assumption 2]{yun2018efficiently}, and \cite{bubeck2020network};
	\item Linear Independence Kink Qualification (LIKQ): \cite[Definition 2.6]{griewank2019relaxing} and \cite[Definition 2]{griewank2016first};
	\item Linearly Independent Activated Data (LIAD): Let the index set $\mathcal{J}_k \coloneqq \{j:\bm{w}_k^\top\bm{x}_j=0\}$. For any fixed $k \in [H]$, the data points $\{\bm{x}_i\}_{i \in \mathcal{J}_k}$ are linearly independent. 
\end{itemize}
\end{Definition}
The general position assumption is from the study of hyperplane arrangement. If the data points are generated from an absolutely continuous probability measure (with respect to the Lebesgue measure), then they are in general position almost surely. The LIKQ is introduced by \citet[Definition 2]{griewank2016first} to ensure an efficient Fr\'echet stationarity test for piecewise differentiable function represented in abs-normal form. See \Cref{sec:abs-norm-form} for a brief introduction. The LIAD condition is natural and equivalent to the subjectivity condition in \Cref{fact:rule}.
Let us present the following result, in which we establish the relationship among SQ and the three other regularity conditions in \Cref{def:other-regularity}.
\begin{Proposition}[Regularity comparison]\label{prop:regularity-relation} For the empirical loss of a shallow ReLU network under \Cref{assu:loss}, we have the following relationship:
\[
\textnormal{general position } \Longrightarrow
\textnormal{ LIKQ } \iff
\textnormal{ LIAD } \Longrightarrow
\textnormal{ SQ}.
\]
\end{Proposition}

We exhibit two examples to show the one-side arrows in \Cref{prop:regularity-relation} are strict.
\begin{Example}[SQ $\nRightarrow$ LIAD]
Let the function $f:\mathbb{R}^4\rightarrow \mathbb{R}$ be given as
\[
f(x,y,z,b)\coloneqq\max\{2y+b,0\}+\max\{2x+2z+b,0\}+\max\{x+y+z+b,0\}-\max\{x-z+b,0\}.
\]
Consider $x=y=z=b=0$. It is easy to verify that SQ is satisfied but not LIAD.
Besides, $f$ is nonconvex, nonsmooth, and non-separable. Neither $f$ nor $-f$ is Clarke regular. But by \Cref{thm:reluMain-CR}, the equality-type subdifferential sum rule still holds.
\end{Example}

\begin{Example}[LIAD $\nRightarrow$ general position]
Let the function $f:\mathbb{R}^3\rightarrow \mathbb{R}$ be given as
\[
f(x,y,b)\coloneqq\max\{-2y+b,0\}+\max\{-y+b,0\}+\max\{x+b,0\}-\max\{y+b,0\}.
\]
Consider $x=y=1$ and $b=-1$. LIAD is satisfied, but the data is not in general position.
\end{Example}

In practice, for data $\bm{x}\in\mathbb{R}^d$, if the features of data  include a discrete-valued component, e.g., $x_1\in\{-1,+1\}$, then the points $\{\bm{x}_i\}_{i=1}^N$ are rarely in general position, as at least half of them must lie in the same affine hyperplane $\{\bm{y}:\bm{e}_1^\top \bm{y} = 1\}$ or $\{\bm{y}:\bm{e}_1^\top \bm{y} = -1\}$. 

\begin{Remark}[$G^L_k$ for general position data]
Besides, if the data points are in general position, we have the following compact representation for $G_k^L$
\[
G^L_k= \sum_{i \in [N]\backslash (\mathcal{I}_k^+\cup\mathcal{I}_k^-)} u_k\rho_i\cdot \mathbf{1}_{\bm{w}_k^\top\bm{x}_i > 0} \cdot\bm{x}_i +\sum_{j\in  \mathcal{I}_k^+} u_k\rho_j\bm{x}_j\cdot[0,1]+
 \sum_{j'\in  \mathcal{I}_k^-} u_k\rho_{j'}\bm{x}_{j'} \cdot \{0,1\}.
\]
	
\end{Remark}

The following corollary concerning the Clarke regularity of all local minimizers could be of independent interest.
\begin{Corollary}\label{prop:lmin-regular}
	If at a point, SQ is satisfied and the empirical loss function $L$ has nonempty Fr\'echet subdifferential here, then the function $L$ is Clarke regular at that point. Consequently, with data in general position, $L$ is Clarke regular at every local minimizer.
\end{Corollary}

\section{Testing of Stationarity Concepts}\label{sec:test}
To perform the stationarity test, we need the following quantitative regularities to characterize the curvature of the pieces in the empirical loss.
\begin{Assumption}
In this section, we further assume that for any $i\in[N]$, the norm of data $\|\bm{x}_i\|_2 \leq R$ and the function $\ell_i$ is $L_\ell$-Lipschitz continuous with an $L_{\ell'}$-Lipschitz continuous gradient $\ell'_i$
\end{Assumption}

\subsection{Exact Stationarity Test}\label{sec:test-exact}

As an immediate illustration of the results in \Cref{sec:regularity}, we record the following exact testing schemes for Clarke and Fr\'echet stationary points. Compared with the developments in \cite{yun2018efficiently} which check the Fr\'echet stationarity from the primal perspective and use polyhedral geometry to avoid redundant computation, by using \Cref{thm:reluMain-CR-frechet}, our treatment for Fr\'echet stationarity is transparent and its correctness is self-evident.
\paragraph{Clarke stationarity.}

\begin{algorithm}[htp]  
	\caption{Exact Stationarity Test (Clarke)}  
	\label{alg:exact}
	\begin{algorithmic}[1]  
		\Procedure{ETest-C}{$u_1, \bm{w}_1, \dots, u_H, \bm{w}_H$, $\bm{x}_1, \dots, \bm{x}_N$} %
		\State compute $\{\rho_i\}_{i=1}^N$, $\mathcal{I}_k^+(\bm{w}_k)$, and $\mathcal{I}_k^-(\bm{w}_k)$ for any $k \in \{1,\dots,H\}$;
		\If{ Span Qualification \eqref{eq:sq} is not satisfied}
		\State \Return{\texttt{not-SQ}};
		\EndIf
		\For{$k\in\{1,\dots,H\}$}
		\State compute $\epsilon_{1,k} \leftarrow \left| \sum_{i=1}^N \rho_i \cdot \max\left\{ \bm{w}_k^\top \bm{x}_i, 0 \right\} \right|$;
		\State compute 
		$\epsilon_{2,k} \leftarrow \dist\left(\bm{0}, G_k^C \right);
		$ \Comment{convex QP}
		\EndFor
		\State \Return{$\sqrt{\sum_{k=1}^H\left(\epsilon_{1,k}\right)^2+\left(\epsilon_{2,k}\right)^2}$;}
		\EndProcedure
	\end{algorithmic}  
\end{algorithm} 

Suppose that SQ is satisfied at the point $(u_1,\bm{w}_1,\dots,u_H,\bm{w}_H)$. By \Cref{thm:reluMain-CR}, it is a Clarke stationarity point of $L$ if and only if, for any $k\in[H]$,
\begin{enumerate}[label=\textnormal{(\alph*)}]
	\item $0=\sum_{i=1}^N \rho_i \cdot \max\left\{ \bm{w}_k^\top \bm{x}_i, 0 \right\}$;
	\item $\bm{0} \in \sum_{i \in [N]\backslash (\mathcal{I}_k^+\cup\mathcal{I}_k^-)} u_k\rho_i\cdot \mathbf{1}_{\bm{w}_k^\top\bm{x}_i > 0} \cdot\bm{x}_i + \sum_{j\in  \mathcal{I}_k^+\cup\mathcal{I}_k^-} u_k\rho_j\bm{x}_j\cdot[0,1]$.
\end{enumerate}
Condition (a) is a simple equality test and condition (b) can be checked by solving a linear programming problem. \Cref{alg:exact} is for testing $\epsilon$-Clarke stationary points.
\paragraph{Fr\'echet stationarity.}

\begin{algorithm}[htp]  
	\caption{Exact Stationarity Test (Fr\'echet)}  
	\label{alg:exact-f}
	\begin{algorithmic}[1]  
		\Procedure{ETest-F}{$u_1, \bm{w}_1, \dots, u_H, \bm{w}_H$, $\bm{x}_1, \dots, \bm{x}_N$} %
		\State compute $\{\rho_i\}_{i=1}^N$, $\mathcal{I}_k^-(\bm{w}_k)$, and $\mathcal{I}_k^+(\bm{w}_k)$ for any $k \in \{1,\dots,H\}$;
		\If{ Span Qualification in \eqref{eq:sq} is not satisfied}
		\State \Return{\texttt{not-SQ}};
		\EndIf
		\For{$k\in\{1,\dots,H\}$}
		\If{$\mathcal{I}_k^-(\bm{w}_k) \neq \emptyset$}
		\State \Return $+\infty$;
		\EndIf
		\State compute $\epsilon_{1,k} = \left| \sum_{i=1}^N \rho_i \cdot \max\left\{ \bm{w}_k^\top \bm{x}_i, 0 \right\} \right|$;
		\State compute 
		$\epsilon_{2,k} = \dist\left(\bm{0}, G_k^F\right);
		$ \Comment{convex QP}
		\EndFor
		\State \Return{$\sqrt{\sum_{k=1}^H\left(\epsilon_{1,k}\right)^2+\left(\epsilon_{2,k}\right)^2}$;}
		\EndProcedure
	\end{algorithmic}  
\end{algorithm} 

	Suppose that SQ is satisfied at the point $(u_1,\bm{w}_1,\dots,u_H,\bm{w}_H)$. By \Cref{thm:reluMain-CR-frechet}, it is a Fr\'echet stationarity point of $L$ if and only if, for any $k\in[H]$,
\begin{enumerate}[label=\textnormal{(\alph*)}]
	\item $0=\sum_{i=1}^N \rho_i \cdot \max\left\{ \bm{w}_k^\top \bm{x}_i, 0 \right\}$;
	\item $\mathcal{I}_k^- = \emptyset$;
	\item $\bm{0} \in \sum_{i \in [N]\backslash (\mathcal{I}_k^+\cup\mathcal{I}_k^-)} u_k\rho_i\cdot \mathbf{1}_{\bm{w}_k^\top\bm{x}_i > 0} \cdot\bm{x}_i + \sum_{j\in  \mathcal{I}_k^+} u_k\rho_j\bm{x}_j\cdot[0,1]$.
\end{enumerate}
Similarly, all above conditions can be checked in polynomial time with \Cref{alg:exact-f}. %

\subsection{Robust Stationarity Test}
In this subsection, we introduce our main algorithmic results. First, we formally define the notion of stationarities that we are aiming to check; see \cite{davis2019stochastic,kornowski2021oracle,tian2022finite} for results on finding near-approximately stationary points for Lipschitz functions.
\begin{Definition}[Near-Approximate Stationarity, NAS]\label{def:nas}
	Given a locally Lipschitz function $f:\mathbb{R}^d\rightarrow\mathbb{R}$, we say that the point $\bm{x}\in\mathbb{R}^d$ is an
	\begin{itemize}
		\item $(\epsilon,\delta)$-Clarke NAS point, if $\dist\Big(\bm{0},\cup_{\bm{y}\in\mathbb{B}_\delta (\bm{x})} \partial_C f(\bm{y})\Big) \leq \epsilon$;
		\item $(\epsilon,\delta)$-Fr\'echet NAS point, if $\dist\Big(\bm{0},\cup_{\bm{y}\in\mathbb{B}_\delta (\bm{x})} \widehat{\partial} f(\bm{y})\Big) \leq \epsilon$.
	\end{itemize}
\end{Definition}

We consider a constructive approach, that is, we certify the $(\epsilon,\delta)$-Clarke NAS of a point $\bm{x}$ for the function $f$ only if we find a point $\bm{y} \in \mathbb{B}_\delta(\bm{x})$ satisfying $\dist(\bm{0}, \partial_C f(\bm{y})) \leq \epsilon$.
Note that, in any time, if a point $\bm{y}\in\mathbb{B}_\delta(\bm{x})$ passes the exact stationarity test, say, with \Cref{alg:exact}, then $\bm{x}$ must be an $(\epsilon,\delta)$-Clarke NAS point. 
In other words, there is no false positive in the test.
The question is that, if $\bm{x}$ is sufficiently closed to a Clarke stationary point, can we always find a point $\bm{y}$ near $\bm{x}$ such that $\bm{y}$ is $\epsilon$-Clarke stationary? That is to say, we need to control the false negative of our robust test. Without exploiting structures in the objective function, finding such a point is impossible in general \cite[Theorem 2.7]{tian2022no}. Our technique is a new rounding scheme (see \Cref{alg:rounding}), which is motivated by the notion of active manifold identification \cite{lewis2002active,lemarechal2000u} in the literature. This new rounding scheme is capable to identify the activation pattern of the target stationary point that $\bm{x}$ is sufficiently close to.

Now, suppose that $f$ is $L$-smooth and a point $\bm{x}^*$ satisfies $\|\nabla f(\bm{x}^*)\| \leq \epsilon$. Without knowing  the concrete structure of $f$, what we can say for any point $\bm{y} \in \mathbb{B}_\delta(\bm{x}^*)$ is that $\|\nabla f(\bm{y})\| \leq \epsilon + L\cdot \delta$, which is the best result we can hope for our test, as we do not assume any concrete structure in the loss $\ell_i$ except their smoothness. Such an estimation cannot hold trivially for a nonsmooth function. Consider $f(x) = |x|$ and $x^* = 0$. For any $\delta > 0$ and $0\neq y \in \mathbb{B}_\delta(x^*)$, we have $|f'(y)| = 1$. 

\subsubsection{Testing Clarke NAS}

\begin{algorithm}[htb]  
	\caption{Neural Rounding (Clarke)}  
	\label{alg:rounding}
	\begin{algorithmic}[1]  
		\Procedure{Rnd-C}{$u_1,\bm{w}_1, \dots, u_H,\bm{w}_H$, $\delta$, $\bm{x}_1, \dots, \bm{x}_N$} %
		\State compute $R = \max_{1 \leq i \leq N} \|\bm{x}_i \|$;
		\For{$k\in\{1,\dots,H\}$}
		\State compute $\widehat{\bm{w}}_k$ by solving the following convex QP %
		\begin{alignat*}{2}
			\widehat{\bm{w}}_k = \argmin_{ \bm{z} \in \mathbb{R}^d} &\ \| \bm{z} - \bm{w}_k \|^2 \\
			\textnormal{s.t.}\ \  &\ \bm{z}^\top \bm{x}_i \geq 2R\cdot\delta, && \forall i\in [N]: \bm{x}_i^\top \bm{w}_k > R\cdot \delta, \\
			&\ \bm{z}^\top \bm{x}_i \leq -2R\cdot\delta, \qquad&& \forall i\in [N]: \bm{x}_i^\top \bm{w}_k < -R\cdot \delta, \\
			&\ \bm{z}^\top \bm{x}_i =0, && \forall i\in [N]: \left|\bm{x}_i^\top \bm{w}_k\right| \leq R\cdot \delta.
		\end{alignat*}
		\EndFor
		\State \Return{$(u_1,\widehat{\bm{w}}_1, \dots, u_H,\widehat{\bm{w}}_H)$;}
		\EndProcedure
	\end{algorithmic}  
\end{algorithm} 

\begin{algorithm}[htb]  
	\caption{Robust Stationarity Test (General)}  
	\label{alg:robust}
	\begin{algorithmic}[1]  
		\Procedure{RTest}{\textsc{ETest}, \textsc{Rnd}, $u_1, \bm{w}_1, \dots, u_H,  \bm{w}_H$, $\delta$, $\bm{x}_1, \dots, \bm{x}_N$} %
		\State $(\widehat{u}_1, \widehat{\bm{w}}_1, \dots, \widehat{u}_H,\widehat{\bm{w}}_H)$ = \textsc{Rnd}$(u_1,\bm{w}_1, \dots, u_H, \bm{w}_H,\delta,\bm{x}_1, \dots, \bm{x}_N)$;
		\If{$\|(u_1,\bm{w}_1, \dots, u_H,\bm{w}_H)-(\widehat{u}_1,\widehat{\bm{w}}_1, \dots, \widehat{u}_H,\widehat{\bm{w}}_H)\| > \delta$}
		\State \Return $+\infty$;
		\EndIf
		\State \Return{\textsc{ETest}$(\widehat{u}_1, \widehat{\bm{w}}_1, \dots, \widehat{u}_H,  \widehat{\bm{w}}_H$, $\bm{x}_1, \dots, \bm{x}_N)$;}
		\EndProcedure
	\end{algorithmic}  
\end{algorithm} 

We define two constants that will be used in the analysis.
\begin{Definition}[Clarke]
	Given a point $(u_1^*, \bm{w}_1^*, \dots, u_H^*,\bm{w}_H^*)$ with a Euclidean norm $B \in [0,+\infty)$, we define the following constants about the separation and curvature of pieces around this point:
	\begin{itemize}
		\item Separation: $
		C_\tau^{\textnormal{Clarke}} \coloneqq \frac{1}{4R} \cdot \min \left\{ \left|\bm{x}_i^\top \bm{w}^*_k \right|: i \in [N], k \in [H], \bm{x}_i^\top \bm{w}_k^* \neq 0 \right\};
		$
		\item Curvature: $C_\mu^{\textnormal{Clarke}} \coloneqq \textnormal{poly}(B,R,L_\ell, L_{\ell'}, N, H)$.\footnote{See \Cref{sec:prf-thm-robust} for the exact value.}
	\end{itemize}
\end{Definition}
\begin{Remark}
If for any $i \in [N]$ and $k \in [H]$, it holds $\bm{x}_i^\top \bm{w}_k^*=0$, then we define the separation constant $C_\tau^{\textnormal{Clarke}} \coloneqq+\infty$, as in the optimization of extended-real-valued functions, $\inf \emptyset = +\infty$.
It is notable that, while the separation constant $C_\tau^{\textnormal{Clarke}}$ is usually unknown when running the testing algorithm, the curvature constant $C_\mu^{\textnormal{Clarke}}$ can be easily estimated when the candidate network and the radius $\delta$ are given.
\end{Remark}

\begin{Theorem}[Robust Clarke test]\label{thm:robust}
Let an $\epsilon$-Clarke stationary point $(u_1^*, \bm{w}_1^*, \dots, u_H^*,\bm{w}_H^*)$ satisfying SQ be given. For any $0 < \delta \leq C_\tau^{\textnormal{Clarke}}$ and any 
\[(u_1, \bm{w}_1, \dots, u_H,\bm{w}_H) \in \mathbb{B}_\delta\big( (u_1^*, \bm{w}_1^*, \dots, u_H^*,\bm{w}_H^*) \big),
\] if the output point $(\widehat{u}_1, \widehat{\bm{w}}_1, \dots, \widehat{u}_H,\widehat{\bm{w}}_H)$ of \Cref{alg:rounding} satisfies SQ, then we have
\[
\dist\Big( \bm{0}, \partial_C L(\widehat{u}_1, \widehat{\bm{w}}_1, \dots, \widehat{u}_H,\widehat{\bm{w}}_H) \Big) \leq \epsilon + C_\mu^{\textnormal{Clarke}}\cdot\delta.
\]
\end{Theorem}

In \Cref{thm:robust}, we show that for a point that is sufficiently closed to an $\epsilon$-Clarke stationary one, and a properly chosen parameter $\delta > 0$, one can correctly certify the near-approximate stationarity of this point in the style as if the function $L$ is smooth by calling \Cref{alg:robust} with $\textsc{RTest}(\textsc{ETest-C}, \textsc{Rnd-C}, \cdots)$. A natural question here is how to choose a proper parameter $\delta$, as the separation constant $C_\tau^{\textnormal{Clarke}}$ is usually unknown. It turns out that a simple line search will work for that.
\begin{Remark}[Line search]
Set the initial value of radius $\delta$ to, say, $\delta_0=1$. Then, in the $t$-th iteration, run \Cref{alg:robust} with parameter $\delta_t$ and set $\delta_{t+1} = \delta_t /2$. Note that for a sufficiently small $\delta$, the rounding scheme in \Cref{alg:rounding} becomes superfluous, as for any $i\in[N]$ and $k\in[H]$ such that $\bm{x}^\top_i \bm{w}_k \neq \bm{0}$, we have $|\bm{x}_i^\top \bm{w}_k| > 2R\cdot \delta$ for a small $\delta$. Therefore, we can stop the line search within at most
\[
\left\lceil\log_2\left( 2R\left/ \min \left\{ \left|\bm{x}_i^\top \bm{w}_k \right|: i \in [N], k \in [H], \bm{x}_i^\top \bm{w}_k \neq \bm{0} \right. \right\} \right) \right\rceil
\]
iterations. It is immediate that, if $(u_1, \bm{w}_1, \dots, u_H,\bm{w}_H) \in \mathbb{B}_{C_\tau^{\textnormal{Clarke}}/2}\big( (u_1^*, \bm{w}_1^*, \dots, u_H^*,\bm{w}_H^*)\big)$, then there exists a radius $\delta_t \in [C_\tau^{\textnormal{Clarke}}/2, C_\tau^{\textnormal{Clarke}}]$ in the iteration sequence such that
\[
\dist\Big( \bm{0}, \partial_C L(\widehat{u}_1, \widehat{\bm{w}}_1, \dots, \widehat{u}_H,\widehat{\bm{w}}_H) \Big) \leq \epsilon + C_\mu^{\textnormal{Clarke}}\cdot\delta_t.
\]
This search scheme also works for the Fr\'echet NAS test and we will not repeat that.
\end{Remark}

\subsubsection{Testing Fr\'echet NAS}

\begin{algorithm}[htp]  
	\caption{Neural Rounding (Fr\'echet)}  
	\label{alg:rounding-f}
	\begin{algorithmic}[1]  
		\Procedure{Rnd-F}{$\bm{w}_1, \dots, \bm{w}_H$, $\delta$, $\bm{x}_1, \dots, \bm{x}_N$} %
		\State compute $R = \max_{1 \leq i \leq N} \|\bm{x}_i \|$ and $C_u=L_{\ell'}(4HRB^2+1)$;
		\For{$k\in\{1,\dots,H\}$}
		\State compute $\widehat{\bm{w}}_k$ by solving the following QP %
		\begin{alignat*}{2}
			\widehat{\bm{w}}_k = \argmin_{ \bm{z} \in \mathbb{R}^d} &\ \| \bm{z} - \bm{w}_k \|^2 \\
			\textnormal{s.t.}\ \  &\ \bm{z}^\top \bm{x}_i \geq 2R\cdot\delta, && \forall i\in [N]: \bm{x}_i^\top \bm{w}_k > R\cdot \delta, \\
			&\ \bm{z}^\top \bm{x}_i \leq -2R\cdot\delta, \qquad&& \forall i\in [N]: \bm{x}_i^\top \bm{w}_k < -R\cdot \delta, \\
			&\ \bm{z}^\top \bm{x}_i =0, && \forall i\in [N]: \left|\bm{x}_i^\top \bm{w}_k\right| \leq R\cdot \delta.
		\end{alignat*}
		\State set $\widehat{u}_k =  u_k$;
		\If{$\min_{i:\bm{w}_k^\top\bm{x}_i = 0}u_k\cdot \rho_i\leq 2C_u\cdot \delta$}
		\State set $\widehat{u}_k = 0$;
		\EndIf
		\EndFor
		\State \Return{$(\widehat{u}_1,\widehat{\bm{w}}_1, \dots, \widehat{u}_H,\widehat{\bm{w}}_H)$;}
		\EndProcedure
	\end{algorithmic}  
\end{algorithm}

Unlike the Clarke case, we need the following extra  nondegeneracy condition on $\ell_i$ to identify the pattern of $\{u_k^*\}_k$ and avoid the Fr\'echet subdifferential being empty.

\begin{Assumption}\label{assu:frechet-nondegenerate}
Given a point $(u_1^*, \bm{w}_1^*, \dots, u_H^*,\bm{w}_H^*)$, we assume that for any $i\in[N]$ such that $\min_{k\in[H]} |\bm{x}_i^\top\bm{w}_k^*| = 0$, we have $\ell_i'\left( \sum_{k=1}^H u_k^*\cdot \max\left\{ (\bm{w}_k^*)^\top \bm{x}_i, 0 \right\} \right) \neq 0$.
\end{Assumption}

The following two constants will be used in the analysis.
\begin{Definition}[Fr\'echet]
	Given a point $(u_1^*, \bm{w}_1^*, \dots, u_H^*,\bm{w}_H^*)$ with a Euclidean norm $B \in [0,+\infty)$, we define two constants concerning the separation and curvature of pieces around this point:
	\begin{itemize}
		\item Separation: $%
		C_{\tau}^{\textnormal{Fr\'echet}} \coloneqq \min\left\{ \min_{\substack{i \in [N], k \in [H],\\ \bm{x}_i^\top \bm{w}_k^* \neq 0 \\}} \frac{\left|\bm{x}_i^\top \bm{w}^*_k \right|}{4R}, \min_{\substack{i \in [N], k \in [H], \\ \bm{x}_i^\top \bm{w}_k^* = 0, u_k^*\cdot \rho_i^* > 0 }} \frac{u_k^*\cdot \rho_i^*}{L_{\ell'}(4HRB^2+1)}\right\};
			$	
		\item Curvature: $C_\mu^{\textnormal{Fr\'echet}} \coloneqq \textnormal{poly}(B,R,L_\ell, L_{\ell'}, N, H)$.\footnote{See \Cref{sec:appd-test-f} for the exact value.}
	\end{itemize}
\end{Definition}

Then, for Fr\'echet NAS test, we have the following result similar to \Cref{thm:robust}.

\begin{Theorem}[Robust Fr\'echet test]\label{thm:robust-f}
Let an $\epsilon$-Fr\'echet stationary point $(u_1^*, \bm{w}_1^*, \dots, u_H^*,\bm{w}_H^*)$ satisfying SQ be given. For any $0 < \delta \leq C_\tau^{\textnormal{Fr\'echet}}$ and any 
\[(u_1, \bm{w}_1, \dots, u_H,\bm{w}_H) \in \mathbb{B}_\delta\big( (u_1^*, \bm{w}_1^*, \dots, u_H^*,\bm{w}_H^*) \big),
\] if the output point $(\widehat{u}_1, \widehat{\bm{w}}_1, \dots, \widehat{u}_H,\widehat{\bm{w}}_H)$ of \Cref{alg:rounding-f} satisfies SQ, then we have
\[
\dist\Big( \bm{0}, \widehat{\partial} L(\widehat{u}_1, \widehat{\bm{w}}_1, \dots, \widehat{u}_H,\widehat{\bm{w}}_H) \Big) \leq \epsilon + C_\mu^{\textnormal{Fr\'echet}}\cdot\delta.
\]
\end{Theorem}

\appendix
\crefalias{section}{appendix}
\section{Abs-Normal Form of Piecewise Differentiable Functions}\label{sec:abs-norm-form}
We briefly review the abs-normal representation of a subclass of piecewise differentiable functions. See \cite{griewank2013stable,griewank2016first} for details.
\subsection{The General Framework}
The abs-normal representation \cite{griewank2013stable} is a piecewise linearization scheme concerning a certain subclass of piecewise differentiable functions in the sense of \citet{scholtes2012introduction}. In this subclass, functions are defined as compositions of smooth functions and the absolute value function. By identifies $\max\{a,b\}=(a+b)/2+|a-b|/2, \min\{a,b\}=(a+b)/2-|a-b|/2,$ and $\max\{x,0\}=x/2+|x|/2$, composition with these nonsmooth elemental functions can also be represented in the abs-normal form.

Let $\varphi:\mathbb{R}^d\rightarrow\mathbb{R}$ be a function in such subclass. By numbering all input to the absolute value functions in the evaluation order as ``switching variables'' $z_i$ for $i\in \{1, \dots, s\}$, the function $\bm{x}\mapsto y=\varphi(\bm{x})$ can be written in the following abs-normal form:
\[
\bm{z} = F(\bm{x}, \bm{p}),\qquad y = f(\bm{x}, \bm{p}),
\]
where $\bm{x}\in\mathbb{R}^d, \bm{p}\in\mathbb{R}_+^s$, the smooth mapping $F:\mathbb{R}^d\times \mathbb{R}^s_+ \rightarrow \mathbb{R}^s$, and the smooth function $f:\mathbb{R}^d\times \mathbb{R}^s \rightarrow \mathbb{R}$. As the numbering of $\{z_i\}_i$ is in the evaluation order, $z_i$ is a function of $z_j$ only if $j < i$. In sum, we have
\[
y = \varphi(\bm{x}) = f(\bm{x}, |\bm{z}(\bm{x})|),
\]
where $\bm{z}(\bm{x})$ a successive evaluation of $\{z_i\}_{i=1}^s$ with given $\bm{x}$. To see such an evaluation of $\bm{z}(\bm{x})$ is well-defined, note that $z_1 = F_1(\bm{x})$ and for any $1 < i \leq s$,
\[ 
z_i = F_i(\bm{x}, |z_1|, \cdots, |z_{i-1}|).
\]
We remark that, similar to the Difference of Convex (DC) decomposition in DC programming, the function $\varphi$ may have many different abs-normal decomposition. The following vectors and matrices are useful when study the function in abs-normal form:
\begin{alignat*}{2}
\bm{a} &\coloneqq \frac{\partial}{\partial \bm{x}} f(\bm{x},\bm{p}) \in \mathbb{R}^d, \qquad
&&\bm{Z} \coloneqq \frac{\partial}{\partial \bm{x}} F(\bm{x},\bm{p}) \in \mathbb{R}^{s\times d}, \\
\bm{b} &\coloneqq \frac{\partial}{\partial \bm{p}} f(\bm{x},\bm{p}) \in \mathbb{R}^s, 
&&\bm{L} \coloneqq \frac{\partial}{\partial \bm{p}} F(\bm{x},\bm{p}) \in \mathbb{R}^{s\times s}.
\end{alignat*}
For any $\bm{\sigma} \in \{-1,1\}^s$, we will denote by $\bm{\Sigma } \coloneqq \diag(\bm{\sigma }) \in \{-1,0,1\}^{s\times s}$. Let us define (see also \cite[Equation (11)]{griewank2016first})
\[
\nabla \bm{z}^\sigma \coloneqq (\bm{I} - \bm{L\Sigma})^{-1}\bm{Z} \in \mathbb{R}^{s\times d},
\]
which will play a key role in the definition of LIKQ (see \Cref{def:detailed-regularity}).
\subsection{Abs-Normal Form of Shallow ReLU Networks}\label{sec:abs-relu}
We rewrite the empirical loss of the shallow ReLU network with absolute value functions as
\[
L(u_1, \bm{w}_1, \cdots, u_H, \bm{w}_H)=\sum_{i=1}^N \ell_i\left( \sum_{k=1}^H \frac{u_k}{2} \cdot \left( \bm{w}_k^\top \bm{x}_i + \left|\bm{w}_k^\top \bm{x}_i\right|  \right) \right).
\]
Then, as there are $N\cdot H$ absolute value evaluations in total, we define the switching variable $\bm{z}\in\mathbb{R}^{NH}$ and the smooth mapping $F$ as
\[
z_{N(k-1)+i} = F_{N(k-1)+i}(u_1, \bm{w}_1, \cdots, u_H, \bm{w}_H) = \bm{w}_k^\top \bm{x}_i, \qquad \forall k \in [H], i \in [N].
\]
The smooth function $f$ in the abs-normal form can be written as
\[
y= f(u_1, \bm{w}_1, \cdots, u_H, \bm{w}_H, \bm{p}) = \sum_{i=1}^N \ell_i\left( \sum_{k=1}^H \frac{u_k}{2} \cdot \left( \bm{w}_k^\top \bm{x}_i + p_{N(k-1)+i}  \right) \right),
\]
where $\bm{p} \in \mathbb{R}_+^{NH}$. Consequently, the matrix $\bm{L}=\bm{0}$, which implies the function $L$ is ``simply switched'' in the sense of \citet{griewank2016first}. For the matrix $\bm{Z}$ and any $k \in [H], i \in [N]$, the $(N(k-1)+i)$-th row of $\bm{Z} \in \mathbb{R}^{NH\times H(d+1)}$ can be written as
\[
\prod_{k'=1}^H 0 \times \bm{1}_{k'=k}\cdot \bm{x}_i^\top \in \mathbb{R}^{1\times H(d+1)}.
\]

\section{Proofs for \Cref{sec:hardness}}

\subsection{The Problems}
\begin{Problem}[3SAT]\label{prob:3sat}
	Given a collection of clauses $\{C_i(\bm{x})\}_{i=1}^n$ on Boolean variables $\bm{x} \in \{0,1\}^m$ such that clause $C_i(\bm{x})$ is limited to a disjunction of at most three literals for any $1 \leq i \leq n$. Let the following formula of $C(\bm{x})$ in conjunctive normal form be given
	\[
	C(\bm{x}) \coloneqq \bigwedge_{i=1}^n  C_i(\bm{x}).
	\]
	Is there an $\bm{x} \in \{0,1\}^m$ satisfying $C(\bm{x}) = 1$? 
\end{Problem}

\begin{Problem}[Piecewise Linear Test, PLT]\label{prob:plt} Suppose $\epsilon \in [0,\frac{1}{\sqrt{m}})$ and 
 the input data $\{\bm{y}_i\}_{i=1}^{3n} \subseteq \mathbb{Z}^{m}$ be given. Let us define a function $f_{\textsf{PLT}}: \mathbb{R}^m \rightarrow \mathbb{R}$ as
	\[
	f_{\textsf{PLT}}(\bm{d})\coloneqq \max_{1\leq i \leq n}- \sum_{j=1}^3 \max\left\{ \bm{d}^\top \bm{y}_{3(i-1)+j},0 \right\}.
	\]
	Is there a vector  $\bm{g} \in \mathbb{R}^m$ satisfying $ \|\bm{g}\| \leq \epsilon$ and 
	\[
	f_{\textsf{PLT}}(\bm{d}) \geq \langle \bm{g}, \bm{d} \rangle, \quad \forall \bm{d} \in \mathbb{R}^m? \tag{PLT}
	\]
	Its complement is given by
	\[
	\forall \bm{g} \in \mathbb{B}_\epsilon(\bm{0}), \exists \bm{d} \in \mathbb{R}^m: f_{\textsf{PLT}}(\bm{d}) < \langle \bm{g}, \bm{d} \rangle. \tag{$\overline{\textnormal{PLT}}$}
	\]
\end{Problem}

\begin{Problem}[Neural Network Test, NNT]\label{prob:nnt} Suppose $\epsilon \in [0, \frac{1}{\sqrt{m}}]$.
	Let the input data $\bm{Y}=\left[ \begin{array}{c|c|c}
		\bm{y}_1 & \cdots & \bm{y}_{3n} \end{array} \right]\subseteq \mathbb{Z}^{m\times 3n}$ be given. Let us define  $f_{\textsf{NNT}}:\mathbb{R}^{3n}\times\mathbb{R}^m \rightarrow \mathbb{R}$ as
	\[
	f_{\textsf{NNT}}(\bm{u}, \bm{w})\coloneqq \max_{1\leq i \leq n} \sum_{j=1}^3 u_{3(i-1)+j} \cdot \max\left\{ \bm{w}^\top \bm{y}_{3(i-1)+j},0 \right\}.
	\]
	Is $(-\mathbf{1}_{3n}, \bm{0}_m)$ an $\epsilon$-Fr\'echet stationary point of $f_{\textsf{NNT}}$, i.e., $\dist\big(\bm{0}, \widehat{\partial} f_{\textsf{NNT}}(-\mathbf{1}_{3n}, \bm{0}_m)\big) \leq \epsilon$?%
\end{Problem}

\begin{Problem}[Abs-Normal Form Test, ANFT]\label{prob:anft}
	Suppose a piecewise linear function is given in the abs-linear form with vectors and matrices $\bm{a}\in\mathbb{R}^n, \bm{b}\in\mathbb{R}^s, \bm{Z} \in \mathbb{R}^{s\times n}, \bm{L}\in\mathbb{R}^{s\times s}$. Is there a definite signature vector $\bm{\sigma} \in \{ -1, 1 \}^s$ such that the following system with respect to $\bm{\mu}_\sigma \in \mathbb{R}^s$ is incompatible 
	\[
	\bm{a}^\top + (\bm{b} - \bm{\mu}_\sigma)^\top \big(\diag(\bm{\sigma}) - \bm{L}\big)^{-1} \bm{Z} = 0, \qquad 0\leq \bm{\mu}_\sigma \in \mathbb{R}^s?
	\]
\end{Problem}

\subsection{Hardness of Piecewise Linear Test}

\begin{Lemma}\label{lem:plt-np-hard}
	\Cref{prob:plt} (PLT) is co-NP-hard.
\end{Lemma}
\begin{proof}
We have to show that $\overline{\textnormal{PLT}}$  is an element of the complexity class NP-hard.
3SAT in \Cref{prob:3sat} is known to be strongly NP-complete \cite{garey1979computers}. We give a polynomial-time reduction from 3SAT to $\overline{\textnormal{PLT}}$. Given any instance of 3SAT, we get clauses $\{C_i(\bm{x})\}_{i=1}^n$ for $\bm{x}\in \{0,1\}^m$. We will refer literals in $C_t(\bm{x})$ by their positions. For example, given $C_t(\bm{x}) = x_i \vee (^\neg x_j) \vee x_k$, we say the literal $x_i$ occurs in $C_t(\bm{x})$ at position $1$, the literal $^\neg x_j$ occurs in $C_t(\bm{x})$ at position $2$, and the literal $x_k$ occurs in $C_t(\bm{x})$ at position $3$. We construct the data $\{\bm{y}_i\}_{i=1}^{3n} \subseteq \mathbb{Z}^m$ as follows
\[
\bm{y}_i =     \left\{ \begin{array}{rcl}
         \bm{e}_k & \mbox{if}
         & \textnormal{Boolean } x_k \textnormal{ occurs in } C_{\lfloor (i-1)/3 \rfloor+1}(\bm{x}) \textnormal{ at position } i - 3\lfloor (i-1)/3 \rfloor \\ 
         -\bm{e}_k & \mbox{if}
         & \textnormal{Boolean } ^\neg x_k \textnormal{ occurs in } C_{\lfloor (i-1)/3 \rfloor+1}(\bm{x}) \textnormal{ at position } i - 3\lfloor (i-1)/3 \rfloor 
                \end{array}\right..
\]
Note the following positive $1$-homogeneous function in the construction of PLT
\[
	f_{\textsf{PLT}}(\bm{d})= \max_{1\leq i \leq n} -\sum_{j=1}^3 \max\left\{ \bm{d}^\top \bm{y}_{3(i-1)+j},0 \right\}.
	\]
	
Suppose that for any $0\leq \|\bm{g}\| \leq \epsilon$, there exists $\bm{d} \in \mathbb{R}^m$ such that $f_{\textsf{PLT}}(\bm{d}) < \langle \bm{g}, \bm{d} \rangle$. We will exhibit an $\bm{x} \in \{0,1\}^m$ such that the given 3SAT is satisfied.
Let $\bm{g} = \bm{0}$ and there  exists $\bm{d} \in \mathbb{R}^m$ such that $f_{\textsf{PLT}}(\bm{d}) < 0$.
 For any $i \in [m]$, let
\[
x_i = \left\{ \begin{array}{rcl}
         1 & \mbox{if}
         & d_i > 0 \\ 
         0 & \mbox{if}
         & d_i \leq 0
                \end{array}\right..
\]
We show $C(\bm{x})=1$. By $f_{\textsf{PLT}}(\bm{d}) < 0$, we get for any $i \in [n]$
\[
\sum_{j=1}^3 \max\left\{ \bm{d}^\top \bm{y}_{3(i-1)+j},0 \right\} > 0,
\]
which implies that there exists a $j' \in \{1,2,3\}$ such that $\bm{d}^\top \bm{y}_{3(i-1)+j'} > 0$. 
Let the index of the Boolean literal occurs in  $C_i(\bm{x})$ at position $j'$ be $k$. Now we consider two cases. If $x_k$ occurs in  $C_i(\bm{x})$ at position $j'$, then $\bm{y}_{3(i-1)+j'} = \bm{e}_k$. We get $\bm{d}^\top \bm{y}_{3(i-1)+j'} = \bm{d}^\top\bm{e}_k = d_k > 0$. So, by definition, $x_k = 1$ which implies $C_i(\bm{x}) = 1$. Otherwise, if $^\neg x_k$ occurs in  $C_i(\bm{x})$ at position $j'$, then $\bm{y}_{3(i-1)+j'} = -\bm{e}_k$. We get $\bm{d}^\top \bm{y}_{3(i-1)+j'} = -\bm{d}^\top\bm{e}_k = -d_k > 0$. So $^\neg x_k = 1$ by definition, which implies $C_i(\bm{x}) = 1$. This shows that $C(\bm{x})=\bigwedge_{i=1}^n  C_i(\bm{x}) = 1$ and the given 3SAT is satisfied.

Conversely, we show that if there exists a vector $\bm{g}$ such that $0\leq \|\bm{g}\| \leq \epsilon$ and  $\inf_{\bm{d}} f_{\textsf{PLT}}(\bm{d}) \geq \langle \bm{g}, \bm{d} \rangle$, then 3SAT cannot be satisfied. Suppose to the contrary that there exists $\bm{x} \in \{0,1\}^m$ such that $C(\bm{x}) = 1$. 
For any $i \in [m]$, let 
\[
d_i = \left\{ \begin{array}{rcl}
         1 & \mbox{if}
         & x_i = 1 \\ 
         -1 & \mbox{if}
         & x_i = 0
                \end{array}\right..
\]
As $\bigwedge_{i=1}^n  C_i(\bm{x}) = 1$, for any $i \in [n]$, there exists a literal of clause $C_i(\bm{x})$ that is satisfied. Let the index of this literal be $k'$ and the position of it in $C_i(\bm{x})$ be $j'$. We consider two cases. If literal $x_{k'}$ occurs in $C_i(\bm{x})$ at position $j'$, then $\bm{y}_{3(i-1)+j'}=\bm{e}_{k'}$. As $C_i(\bm{x})=1$ due to literal $x_{k'}$, we get  $x_{k'} = 1$ and $d_{k'}=1$ by definition. Then, for such $i\in[n]$, we get
\[
\sum_{j=1}^3 \max\left\{ \bm{d}^\top \bm{y}_{3(i-1)+j},0 \right\} \geq \max\left\{ \bm{d}^\top \bm{y}_{3(i-1)+j'},0 \right\} = 
  \max\{d_{k'},0\} = 1.
\] 
Otherwise, if literal $^\neg x_{k'}$ occurs in $C_i(\bm{x})$ at position $j'$, then $\bm{y}_{3(i-1)+j'}=-\bm{e}_{k'}$. As $C_i(\bm{x})=1$ due to literal $^\neg x_{k'}$, we get  $x_{k'} = 0$ and $d_{k'}=-1$ by definition. Then, for any $i\in[n]$, we get
\[
\sum_{j=1}^3 \max\left\{ \bm{d}^\top \bm{y}_{3(i-1)+j},0 \right\} \geq \max\left\{ \bm{d}^\top \bm{y}_{3(i-1)+j'},0 \right\} = 
  \max\{-d_{k'},0\} = 1.
\] 
This gives 
\[
\langle \bm{g}, \bm{d} \rangle \leq f_{\textsf{PLT}}(\bm{d}) \leq -1 < -\epsilon\cdot \sqrt{m} \leq -\|\bm{g}\|\cdot \|\bm{d}\| \leq -|\langle \bm{g}, \bm{d} \rangle| \leq \langle \bm{g}, \bm{d} \rangle, 
\]
a contradiction. Hence \Cref{prob:plt} is in the class co-NP-hard.
\end{proof}

While it is not clear whether the \Cref{prob:plt} with a positive $\epsilon$ is an element of the complexity class co-NP, we show that, when $\epsilon = 0$, \Cref{prob:plt} is in co-NP.

\begin{Lemma}\label{lem:co-NP}
	If $\epsilon = 0$, then \Cref{prob:plt} is in the complexity class of co-NP.
\end{Lemma}
\begin{proof}
For $\epsilon = 0$, we only need to test $f_{\textsf{PLT}}(\bm{d}) \geq 0,  \forall \bm{d} \in \mathbb{R}^m$.
	Given any $\bm{d} \in \mathbb{R}^m$ checking whether $f_{\textsf{PLT}}(\bm{d}) < 0$ can be done in $O(mn\log n)$ time. If the answer to \Cref{prob:plt} is yes, by homogeneity in  $f_{\textsf{PLT}}$,  there exist a direction $\bm{d}$ and a vector $\bm{s} \in \{1,2,3\}^n$ such that $f_{\textsf{PLT}}(\bm{d}) \leq -1$ and $\bm{d}^\top \bm{y}_{3(i-1)+s_i} \geq 1$ for any $i \in [n]$.
	There are only $3^n$ elements in the set $\{1,2,3\}^n$ and all resulting  $\left[ \begin{array}{c|c|c}
		\bm{y}_{s_1} & \cdots & \bm{y}_{3n-3+s_n} \end{array} \right]$ are integer matrix of polynomial length relative to the input size of \Cref{prob:plt}. So the certificate $\bm{d}$ can be obtained by solving a linear program in polynomial time. Therefore, if there exists $\bm{d} \in \mathbb{R}^m$ such that $f_{\textsf{PLT}}(\bm{d}) < 0$, then a nondeterministic algorithm can find $\bm{s} \in \{1,2,3\}^n$ and $\bm{d}'\in \mathbb{Q}^m$ satisfying $f_{\textsf{PLT}}(\bm{d}') \leq -1 < 0$ in polynomial time. Thus, \Cref{prob:plt} with $\epsilon = 0$ is an element of the complexity class co-NP.
\end{proof}

\begin{proof}[Proof of \Cref{thm:hard-general}]
	We first note that \Cref{prob:plt} can be written in the standard max-min form in polynomial time by the following elementary identify:
	\[
	-\sum_{i=1}^3 \max\{t_i, 0\} = \min\left\{\sum_{i=1}^3 s_i\cdot t_i : s_k \in \{-1,0\}, \forall k \in \{1,2,3\} \right\}.
	\]
	Besides, it holds $f_{\textsf{PLT}}(\bm{d}) = f_{\textsf{PLT}}(\bm{0}) + f_{\textsf{PLT}}'(\bm{0}; \bm{d}) = f_{\textsf{PLT}}'(\bm{0}; \bm{d})$.
	By \Cref{def:subd-f}, we know $\dist\big(\bm{0}, \widehat{\partial} f_{\textsf{PLT}}(\bm{0})\big) \leq \epsilon$ if and only if there
	exists a vector  $\bm{g} \in \mathbb{R}^m$ satisfying $0 \leq \|\bm{g}\| \leq \epsilon$ and $f_{\textsf{PLT}}(\bm{d}) \geq \langle \bm{g}, \bm{d} \rangle, \forall \bm{d} \in \mathbb{R}^m$, which is the definition of \Cref{prob:plt}.
	Note that if $\epsilon = 0$, in the reduction from 3SAT in the proof of \Cref{lem:plt-np-hard}, all numerical parameters are bounded by a polynomial of the input size.
	 The proof completes by \Cref{lem:plt-np-hard}.%
\end{proof}

\subsection{Hardness of Abs-Normal Form Test}

\begin{proof}[Proof of \Cref{coro:abs-norm-hard}]
We first show that PLT in \Cref{prob:plt} can be written in the abs-normal form in polynomial time.
For ease of notation, let $q_i(\bm{d}) \coloneqq -\sum_{j=1}^3 \max\left\{ \bm{d}^\top \bm{y}_{3(i-1)+j},0 \right\}$ for any $i \in [n]$. Then, we can rewrite every $q_i$ in the abs-linear form as
\begin{alignat*}{2}
z_i(\bm{d}) &= \bm{y}_i^\top \bm{d}, &\forall i \in [n]. \\	
q_i(\bm{d}, \bm{p}) &= -\frac{1}{2}\sum_{j=1}^3\bm{d}^\top \bm{y}_{3(i-1)+j} - \frac{1}{2}\sum_{j=1}^3 p_{3(i-1)+j}, \qquad &\forall i \in [n].
\end{alignat*}
Note that the function $f_{\textsf{PLT}}$ can be expressed as
\[
y\coloneqq f_{\textsf{PLT}}(\bm{d})=\max_{1 \leq i \leq n} q_i(\bm{d}, |\bm{z}|) = \max\{\dots, \max\{q_1(\bm{d},|\bm{z}|),q_2(\bm{d},|\bm{z}|)\},\dots,q_n(\bm{d},|\bm{z}|)\},
\]
which can be written in abs-normal form as
\begin{alignat*}{2}
	z_i &= F_i(\bm{q}, |\bm{z}|) = \frac{1}{2^{i-2}}\cdot q_1 + \sum_{t=2}^{i-1}\big( q_t + p_{3n+t-1} \big) - q_i, &\qquad\forall 3n+1 \leq i \leq 4n-1,\\
	y&=f(\bm{q}, \bm{p}) = \frac{1}{2^{n-1}}\cdot q_1 + \sum_{t=2}^n\big( q_t + p_{3n+t-1} \big).
\end{alignat*}
In sum, we have
\[
z_i = \left\{ \begin{array}{rcl}
         \displaystyle \bm{y}_i^\top \bm{w} & \mbox{for}
         & 1 \leq i \leq 3n \\ 
         \displaystyle\frac{1}{2^{i-3n-1}}\cdot q_1+\sum_{t=2}^{i-3n-1} \frac{1}{2^{i-3n-t}}\cdot \big( q_t + p_{3n+t-1} \big)  & \mbox{for} & 3n+1 \leq i \leq 4n-1                \end{array}\right..
\]
Then, we know
\[
f_{\textsf{PLT}}(\bm{d})=f(\bm{d}, |\bm{z}(\bm{d})|) = \frac{1}{2^{n-1}}\cdot q_1 + \sum_{t=2}^n\big( q_t + |z_{3n+t-1}(\bm{d})| \big).
\]
Then, the matrices $\bm{L}, \bm{Z}, \bm{a}, \bm{b}$ can be computed in polynomial time.

We note that $\inf_{\bm{d}} f_{\textsf{PLT}} (\bm{d}) \geq 0$ if and only 
if the function $f_{\textsf{PLT}}$ is first-order minimal in abs-normal form and this is shown in the discussion below \cite[Equation (2)]{griewank2019relaxing} (see also \cite[p3]{griewank2016first}). Then, the answer of ANFT in \Cref{prob:anft} for the abs-normal form of $f_{\textsf{PLT}}$ is No if and only if $\bm{0}$ is a Fr\'echet stationary point of $f_{\textsf{PLT}}$. Then, by \Cref{lem:plt-np-hard}, ANFT in \Cref{prob:anft} is NP-hard. To see ANFT is in NP, for any given $\bm{\sigma} \in \{-1,1\}^s$, the computation of the vector $\bm{a}^\top + \bm{b}^\top \big(\diag(\bm{\sigma}) - \bm{L}\big)^{-1}\bm{Z}$ and the matrix $\big(\diag(\bm{\sigma}) - \bm{L}\big)^{-1}\bm{Z}$ can be done in polynomial time. Then, ANFT for a given $\bm{\sigma}$ reduces to check the infeasibility of a linear system, which is in P. In sum, we have shown ANFT in \Cref{prob:anft} is NP-complete, which implies a general test of FOM without kink qualification in \cite[Theorem 4.1]{griewank2019relaxing} is co-NP-complete.
\end{proof}

\subsection{Hardness of Neural Network Test}

\begin{Lemma}\label{lem:nnt-co-np-hard}
	\Cref{prob:nnt} (NNT) is co-NP-hard. If $\epsilon=0$, \Cref{prob:nnt} is co-NP-complete.
\end{Lemma}
\begin{proof}
We first prove that $(-\mathbf{1}_{3n}, \mathbf{0}_m)$ is an $\epsilon$-Fr\'echet stationary point of $f_{\textsf{NNT}}$ if and only if there exists $\bm{g}^w \in \mathbb{B}_\epsilon^m(\bm{0})$ such that $\inf_{\bm{d}\in\mathbb{R}^m} f_{\textsf{PLT}}(\bm{d}) \geq \langle \bm{g}^w, \bm{d}\rangle$ with the same input data $\{\bm{y}_i\}_{i=1}^{3n} \subseteq \mathbb{Z}^{m}$. By \cite[Exercise 8.4]{rockafellar2009variational} and $f_{\textsf{NNT}}$ is B-differentiable; see \cite[Definition 4.1.1]{cui2021modern}, we get $\dist\big(\bm{0}, \widehat{\partial} f_{\textsf{NNT}}(-\mathbf{1}_{3n}, \bm{0}_m)\big) \leq \epsilon$ if and only if there exists $(\bm{g}^u, \bm{g}^w) \in \mathbb{B}_\epsilon^{3n+m}(\bm{0}_{3n+m})$ such that
\[\label{eq:prf-nnt}
f_{\textsf{NNT}}'(-\bm{1}_{3n}, \bm{0}_m; \bm{d}^u, \bm{d}^w)\geq \langle \bm{d}^u, \bm{g}^u \rangle + \langle \bm{d}^w, \bm{g}^w \rangle,\qquad \forall \bm{d}^u \in\mathbb{R}^{3n}, \bm{d}^w \in \mathbb{R}^m. \tag{$\sharp$}
\]
Using the chain rule of directional derivative for B-differentiable function \cite[Proposition 4.1.2(a)]{cui2021modern}, we have
\[
f_{\textsf{NNT}}'(-\bm{1}_{3n}, \bm{0}_m; \bm{d}^u, \bm{d}^w)= 
\max_{1\leq i \leq n}- \sum_{j=1}^3 \max\left\{ \bm{y}_{3(i-1)+j}^\top \bm{d}^w,0 \right\}= f_{\textsf{PLT}}(\bm{d}^w).
\]
For any $\bm{g}^u,\bm{g}^{w}$, consider $\bm{d}^w = \bm{0}$ and $\bm{d}^u = \bm{g}^u$. We get that \Cref{eq:prf-nnt} holds if and only if  $\bm{g}^u = \bm{0}_{3m}$ and $\inf_{\bm{d}^w\in\mathbb{R}^m} f_{\textsf{PLT}}(\bm{d}^w) \geq \langle \bm{g}^w, \bm{d}^w\rangle$,
which completes the proof by the co-NP-hardness of \Cref{prob:plt} in \Cref{lem:plt-np-hard} and co-NP-completeness if $\epsilon = 0$ in  \Cref{lem:co-NP}.
\end{proof}

\begin{proof}[Proof of \Cref{coro:nnhard}]
Note that \Cref{prob:nnt} can be represented by the empirical loss of a convolutional neural network with $N=1$ and architecture
	\[
	 \ell_1\Big(\quad\textnormal{\tt{max-pooling}}\quad\circ\quad \textnormal{\tt{conv}}^{\bm{u}}\quad\circ\quad \textnormal{\tt{ReLU}}\quad \circ\quad\textnormal{\tt conv}^{\bm{w}}(\bm{Y})\quad\Big),
	\]
	where $\ell_1(t)=t$ and $\bm{Y}=\left[ \begin{array}{c|c|c}
		\bm{y}_1 & \cdots & \bm{y}_{3n} \end{array} \right]\subseteq \mathbb{Z}^{m\times 3n}$.
If $\epsilon = 0$, in the reduction from 3SAT to PLT, then to NNT, all numerical parameters are bounded by a polynomial of the input size. The proof completes by \Cref{lem:nnt-co-np-hard}.
\end{proof}

\section{Proofs for \Cref{sec:regularity}}

\subsection{Proof Roadmap}\label{sec:roadmap}
Recall the loss function $L$ of shallow ReLU neural network:
\[
L(u_1,\bm{w}_1,\dots,u_H,\bm{w}_H) \coloneqq \sum_{i=1}^N \ell_i \left( \sum_{k=1}^H u_k\cdot \max\left\{ \bm{w}_k^\top \bm{x}_i, 0 \right\} \right).
\]
Set constants $\rho_i \coloneqq \ell_i'\left( \sum_{k=1}^H u_k\cdot \max\left\{ \bm{w}_k^\top \bm{x}_i, 0 \right\} \right) $ for any $i \in [N]$. Let us first consider a partially linearized loss function $\overline{L}$ defined by
\[
\begin{aligned}
\overline{L}(u_1,\bm{w}_1,\dots,u_H,\bm{w}_H) &\coloneqq \sum_{i=1}^N \rho_i \cdot \left( \sum_{k=1}^H u_k\cdot \max\left\{ \bm{w}_k^\top \bm{x}_i, 0 \right\} \right) \\
&= \sum_{k=1}^H\overline{L}_k(u_k, \bm{w}_k)\coloneqq\left( u_k \cdot \sum_{i=1}^N \rho_i \cdot \max\left\{ \bm{w}_k^\top \bm{x}_i, 0 \right\} \right).
\end{aligned}
\]
By exploiting the smoothness of $\{\ell_i\}_{i=1}^N$ and a Lagrange scalarization technique in \Cref{thm:linearization}, we will show that
\[
\partial_\triangleleft L(u_1,\bm{w}_1,\dots,u_H,\bm{w}_H) = \partial_\triangleleft \overline{L}(u_1,\bm{w}_1,\dots,u_H,\bm{w}_H).
\]
Then, we focus on the linearized $\overline{L}$. By separation of $\{(u_k, \bm{w}_k)\}_k$ and using again the Lagrange scalarization technique in form of \Cref{coro:prod}, we have
\begin{align*}
\partial_\triangleleft \overline{L}(u_1,\bm{w}_1,\dots,u_H,\bm{w}_H) &\overset{(a)}{=}
\prod_{k=1}^H \partial_\triangleleft \overline{L}_k(u_k, \bm{w}_k) \\
&\overset{(b)}{=} \prod_{k=1}^H \left\{  \sum_{i=1}^N \rho_i \cdot \max\left\{ \bm{w}_k^\top \bm{x}_i, 0 \right\} \right\} \times \partial_\triangleleft \Big[ \overline{L}_k(u_k, \cdot) \Big](\bm{w}_k), %
\end{align*}
where (a) is due to \cite[Proposition 2.5]{rockafellar1985extensions} and \cite[Proposition 10.5]{rockafellar2009variational}; (b) is by \Cref{coro:prod}.
Therefore, it holds
\[
\partial_\triangleleft L(u_1,\bm{w}_1,\dots,u_H,\bm{w}_H)
=
\prod_{k=1}^H \left\{   \sum_{i=1}^N \rho_i \cdot \max\left\{ \bm{w}_k^\top \bm{x}_i, 0 \right\}  \right\} \times \partial_\triangleleft \Big[ \overline{L}_k(u_k, \cdot) \Big](\bm{w}_k),
\]
which implies that the validity of exact chain rule of $L$ rely on a careful study of $\overline{L}_k(u_k, \cdot)$.
In particular, if we have the exact chain rule for any $k\in[H]$ as follows
\begin{equation}\label{eq:Lk-CR}
	\partial_\triangleleft \Big[ \overline{L}_k(u_k, \cdot) \Big](\bm{w}_k) =G_k^\triangleleft,
\end{equation}
then we get the validity of exact chain rule for $L$. That is $\partial_\triangleleft L(u_1,\bm{w}_1,\dots,u_H,\bm{w}_H) = $
\[
\prod_{k=1}^H \left\{  \sum_{i=1}^N \rho_i \cdot \max\left\{ \bm{w}_k^\top \bm{x}_i, 0 \right\} \right\} \times 
G_k^\triangleleft.
\]
To prove \Cref{eq:Lk-CR}, we need a fine-grained analysis of $\overline{L}_k(u_k, \cdot)$. First, we isolate the nonsmooth part out by rewritting
\[
\Big[\overline{L}_k(u_k, \cdot) \Big] (\bm{w}_k)=\sum_{i \in [N]\backslash (\mathcal{I}_k^+\cup\mathcal{I}_k^-)} u_k\rho_i\cdot \max\left\{ \bm{w}_k^\top \bm{x}_i, 0 \right\} + f_k(\bm{w}_k),
\]
where we define
\[
f_k(\bm{w}_k)\coloneqq \sum_{i \in \mathcal{I}_k^+} u_k\rho_i\cdot \max\left\{ \bm{w}_k^\top \bm{x}_i, 0 \right\} - \sum_{j \in \mathcal{I}_k^-} (-u_k\rho_j)\cdot \max\left\{ \bm{w}_k^\top \bm{x}_j, 0 \right\}.
\]
What remaining is to study the subdifferential of this non-separable piecewise linear function $f_k$ for any $k\in[H]$ and figure out conditions, under which
\[
\partial_\triangleleft f_k(\bm{w}_k) = \widetilde{G}_k^\triangleleft \coloneqq G_k^\triangleleft - \sum_{i \in [N]\backslash (\mathcal{I}_k^+\cup\mathcal{I}_k^-)} u_k\rho_i\cdot \mathbf{1}_{\bm{w}_k^\top\bm{x}_i > 0} \cdot\bm{x}_i.
\]
This will be done in \Cref{sec:prf-CR-pl}.

\subsection{Technical Lemmas}

\begin{Lemma}[Gordan, {cf.~\cite[Exercise 4.26]{bertsimas1997introduction}}]\label{lem:gordan}
	Let $\bm{A}\in\mathbb{R}^{n\times m}$ be given. Then, exactly one of the following statements is true:
	\begin{itemize}
		\item There exists an $\bm{x}\in \mathbb{R}^m$ such that $\bm{Ax} < \bm{0}$.
		\item There exists a $\bm{y} \in \mathbb{R}^n$ such that $\bm{A}^\top \bm{y} = \bm{0}$ with $\bm{y} \geq \bm{0}, \bm{y}\neq \bm{0}$.
	\end{itemize}
\end{Lemma}

\begin{Lemma}\label{lem:localization}
	Let $A, B, C $ be sets in $\mathbb{R}^n$. Suppose further that $A$ is convex and closed, and $C$ is nonempty and bounded. If the strict inclusion $A \subsetneq B$ holds, then we can assert $A+C \subsetneq B + C$.
\end{Lemma}
\begin{proof}
	Let $\bm{x}_b \in B\backslash A$. The claim is trivial when $A=\emptyset$. Choose $\bm{x}_a' \in A$ and set $\delta\coloneqq \|\bm{x}_b-\bm{x}_a'\|$. As $A$ is closed, the following $\bm{x}_a$ is well-defined
	\[
	\bm{x}_a \coloneqq \argmin_{\bm{a} \in A} \|\bm{a} - \bm{x}_b \|=\argmin_{\bm{a} \in A\cap\mathbb{B}_{\delta}(\bm{x}_b)} \|\bm{a} - \bm{x}_b \|. 
	\]
	Let $\bm{d} \coloneqq \bm{x}_b - \bm{x}_a$. As $\bm{x}_b \notin A$ and $A$ is closed, we know $\|\bm{d}\| > 0$.
	By the optimality condition and convexity of $A$, we know $\langle \bm{a} - \bm{x}_a, \bm{d}\rangle \leq 0, \forall \bm{a}\in A$, which implies $\langle  \bm{d}, \bm{a} \rangle \leq \langle\bm{d}, \bm{x}_a \rangle, \forall \bm{a} \in A$. As $C$ is bounded, we know $\langle \bm{c}, \bm{d} \rangle \leq \|\bm{c}\|\cdot \|\bm{d}\| < +\infty,\forall \bm{c} \in C$. Let $\bm{x}_c $ be
	\[
	\bm{x}_c \in \eargmax_{\bm{c} \in C}\ \langle \bm{c}, \bm{d} \rangle,
	\]
	where $0<\epsilon<\|\bm{d}\|^2$.
	We claim $\bm{x}_b + \bm{x}_c \notin A + C$. Suppose not. Therefore, there exist $\bm{y}_a \in A, \bm{y}_c \in C$ such that $\bm{y}_a + \bm{y}_c = \bm{x}_b + \bm{x}_c$. However, we compute
	\[
	\begin{aligned}
		\langle \bm{d}, \bm{x}_b + \bm{x}_c \rangle &=  \langle \bm{d}, \bm{d} + \bm{x}_a \rangle + \langle \bm{d}, \bm{x}_c \rangle \\
		& \geq \|\bm{d}\|^2 + \langle\bm{d}, \bm{y}_a \rangle + \langle\bm{d}, \bm{y}_c \rangle - \epsilon\\
		& >  \langle\bm{d}, \bm{y}_a + \bm{y}_c \rangle,
	\end{aligned}
	\]
	which gives the contradiction.
\end{proof}
\begin{Remark}
Though the claim seems straightforward,
\Cref{lem:localization} is indeed non-trivial. We record the following counterexamples when different conditions are removed.
\begin{itemize}
	\item $C$ is empty: $A+C = B + C = \emptyset$. 
	\item $C$ is unbounded: if $C=\mathbb{R}^n$ and $A,B$ are nonempty, then $A+C = B + C = \mathbb{R}^n$.
	\item $A$ is nonconvex: if $A=\mathbb{B}\backslash\mathbb{B}_{1/4}, B=\mathbb{B}, C=\mathbb{B}$, then $A+C = B + C = \mathbb{B}_2$.
	\item $A$ is not closed: if $A=\mathbb{B}^\circ, B=\mathbb{B}, C = \mathbb{B}^\circ$, then $A+C = B + C = \mathbb{B}^\circ_2$.
\end{itemize}	
\end{Remark}

\begin{Lemma}\label{lem:extremept}
Let $\{\bm{p}_i\}_{i=1}^n$ be linearly independent. Define a convex set $C = \sum_{i=1}^n \bm{p}_i\cdot[0,1]$. For any $\bm{s} \in \{0,1\}^n$, the point $\bm{p} = \sum_{i=1}^n s_i\cdot\bm{p}_i$ is an extreme point of $C$.
\end{Lemma}
\begin{proof}
	Suppose not and $\bm{p} = \frac{1}{2}\bm{x}_1 + \frac{1}{2}\bm{x}_2 = \sum_{i=1}^n s_i\cdot\bm{p}_i$ with $\bm{p} \neq \bm{x}_1 = \sum_{i=1}^n\alpha_i\cdot \bm{p}_i \in C$ and 
	$\bm{p} \neq \bm{x}_2 = \sum_{i=1}^n\beta_i\cdot \bm{p}_i \in C$. We know $\alpha_i \in [0,1]$ and $\beta_i \in [0,1]$ for any $i\in[n]$ by definition. Thus, it holds
	\[
	\sum_{i=1}^n s_i \cdot \bm{p}_i = \sum_{i=1}^n \left(\frac{\alpha_i + \beta_i}{2} \right)\cdot \bm{p}_i.
	\]
	As $\{\bm{p}_i\}_{i=1}^n$ are linearly independent, we know that, for any $i \in [n]$, it holds $s_i = \left(\frac{\alpha_i + \beta_i}{2} \right) \in \{0,1\}$. If $s_i = 0$, we have $\alpha_i = \beta_i = 0$. Meanwhile, we know $\alpha_i = \beta_i = 1$ if $s_i = 1$. Therefore, it holds $\bm{x}_1 = \bm{x}_2 = \bm{p}$, a contradiction.
\end{proof}

\begin{Lemma}\label{lem:frechet-empty}
Let a function $g:\mathbb{R}^d \rightarrow \mathbb{R}$ be $\bm{w} \mapsto -\sum_{j=1}^m \max\{\bm{y}_j^\top \bm{w}, 0\}$. If there exists $j \in [m]$ such that $\bm{w}^\top \bm{y}_j = 0$ and $\bm{y}_j \neq \bm{0}$, then we have $\widehat{\partial} g(\bm{w}) = \emptyset$.
\end{Lemma}
\begin{proof}
	Suppose not and let $\bm{u} \in \widehat{\partial} g(\bm{w})$. We write
	\[
	g(\bm{w}) = -\sum_{j:\bm{w}^\top \bm{y}_j \neq 0} \max\{\bm{w}^\top \bm{y}_j,0\} +g_0(\bm{w}),
	\]
	where we define $g_0(\bm{w})\coloneqq -\sum_{k:\bm{w}^\top \bm{y}_k = 0} \max\{\bm{w}^\top \bm{y}_k,0\} $. Then, by \cite[Exercise 8.8(c)]{rockafellar2009variational}, we have 
	\[
	\widehat{\partial} g(\bm{w}) = -\sum_{j:\bm{w}^\top \bm{y}_j \neq 0} \mathbf{1}_{\bm{w}^\top \bm{y}_j > 0} \cdot \bm{y}_j + \widehat\partial g_0(\bm{w}).
	\]
	Let $\bm{u}' = \bm{u} + \sum_{j:\bm{w}^\top \bm{y}_j \neq 0} \mathbf{1}_{\bm{w}^\top \bm{y}_j > 0} \cdot  \bm{y}_j$ and we know $\bm{u}' \in \widehat\partial g_0(\bm{w})$. By \cite[Exercise 8.4]{rockafellar2009variational}, for any $\bm{d}\in\mathbb{R}^d$, it holds
	\[
	\bm{u}'^\top \bm{d} \leq g_0'(\bm{w};\bm{d}) = -\sum_{k:\bm{w}^\top \bm{y}_k = 0} \mathbf{1}_{\bm{d}^\top \bm{y}_k>0} \cdot \bm{d}^\top \bm{y}_k \leq 0.
	\] 
	Let $\bm{d}=\bm{u}'$ and we know $\|\bm{u}'\|^2 \leq g_0'(\bm{w};\bm{u}') \leq 0$. Thus, $\bm{u}'=\bm{0}$. Let $\bm{d}$ be any $\bm{y}_j$ such that $\bm{w}^\top \bm{y}_j = 0$ and $\bm{y}_j \neq 0$. Then, we have
	\[
	0 = \bm{u}'^\top \bm{y}_j \leq g_0'(\bm{w};\bm{y}_j) \leq - \|\bm{y}_j\|^2 < 0,
	\]
	a contradiction.
\end{proof}

\begin{Definition}[Bouligand subdifferential, c.f.~{\cite[Definition 4.3.1]{cui2021modern}}]
	\label{def:subd-g}
	Given a point $\bm{x}$, the Bouligand subdifferential of a locally Lipschitz function $f$ at $\bm{x}$ is defined by
	\[
	\partial_B f (\bm{x}) := \big\{ \bm{g}: \exists \{\bm{x}_\nu\} \rightarrow \bm{x} \textnormal{ and } \{\nabla f(\bm{x}_\nu)\} \rightarrow \bm{g} \textnormal{ s.t. } \nabla f(\bm{x}_\nu) \textnormal{ exists for any }\nu \big\}.
	\]

\end{Definition}

\subsection{Partial Linearization via Lagrange Scalarization}
The following theorem is a powerful and general principle.%
\begin{Theorem}[Partial linearization]\label{thm:linearization}
	Let a point $\bm{x}\in\mathbb{R}^d$ and a locally Lipschitz  $f:\mathbb{R}^d \rightarrow \mathbb{R}$ be given in form of composition $f(\bm{x}) = h\circ G(\bm{x})$, where the gradient of $h:\mathbb{R}^n \rightarrow \mathbb{R}$ is locally Lipschitz near $G(\bm{x})$ and $G:\mathbb{R}^d\rightarrow\mathbb{R}^n$ is locally Lipschitz near $\bm{\bm{x}}$. Suppose $h$ and $G$ are directionally differentiable. Then, we have
	\[
	\partial_\triangleleft f(\bm{x}) = \partial_\triangleleft \Big[\left\langle \nabla h\big(G(\bm{x})\big), G(\cdot ) \right\rangle \Big](\bm{x}).
	\]
\end{Theorem}
\begin{proof} Let the partially linearized $f$ at $\bm{x}$ be $\bar{f}:\mathbb{R}^d\rightarrow\mathbb{R}$ defined as
	\[
	\bar{f}(\bm{y})\coloneqq \left\langle \nabla h\big(G(\bm{x})\big), G(\bm{y} ) \right\rangle.
	\]
	For the limiting subdifferential version, the claim directly follows from a margin function chain rule \cite[Theorem 6.5]{mordukhovich1996nonsmooth}. The Clarke subdifferential version directly follows from the relation between Clarke and limiting subdifferential \cite[Theorem 8.49]{rockafellar2009variational} and \cite[Theorem 6.5]{mordukhovich1996nonsmooth}. However, as the proof of \cite[Theorem 6.5]{mordukhovich1996nonsmooth} uses a perturbation argument to approximate $\partial f$ with $\epsilon$-Fr\'echet subdifferential, the machinery is somehow complicated. Here
	we give an elementary proof for the Clarke version from the primal perspective using tools from convex analysis. We show $f^\circ(\bm{x}; \bm{v}) = \bar{f}^\circ(\bm{x}; \bm{v})$ for any $\bm{v}\in\mathbb{R}^d$. Note that the Clarke generalized subderivative can be written as
\[
\begin{aligned}
f^\circ(\bm{x};\bm{v}) &= \limsup_{\substack{\bm{x}'\rightarrow \bm{x} \\ t \searrow 0}} \Delta_t f(\bm{x}')(\bm{v})
 \\
&= \lim_{\epsilon\searrow 0} \sup_{\|\bm{x}'- \bm{x}\|\leq \epsilon} \sup_{0 < t < \epsilon} \Delta_t f(\bm{x}')(\bm{v}),
\end{aligned}
\]
where the difference quotient function $\Delta_t f(\bm{x}'):\mathbb{R}^d \rightarrow\mathbb{R}$ of $f$ at $\bm{x}'$ and direction $\bm{v}$ is defined by
\[
\Delta_t f(\bm{x}')(\bm{v})\coloneqq \frac{f(\bm{x}'+t\bm{v}) - f(\bm{x}')}{t}.
\]
We assume $h$ is $L_h$-smooth near $g(\bm{x})$ and $G$ is $L_G$-Lipschitz near $\bm{x}$. We will use the following estimation (see \cite[Lemma 1.2.3]{nesterov2003introductory}) if $h$ is $L_h$-smooth at $\bm{z} \in \mathbb{R}^n$:
\[
-\frac{L_h}{2}\|\bm{z}'-\bm{z}\|^2 \leq h(\bm{z}') - h(\bm{z})- \left\langle \nabla h(\bm{z}), \bm{z}'-\bm{z}\right\rangle
\leq \frac{L_h}{2}\|\bm{z}'-\bm{z}\|^2.
\]
To prove $f^\circ(\bm{x};\bm{v}) \geq \bar{f}^\circ(\bm{x};\bm{v})$, we compute as follows
\begin{align*}
	\Delta_t f(\bm{x}')(\bm{v})&= \frac{h\big(G(\bm{x}'+t\bm{v})\big) - h\big(G(\bm{x}')\big)}{t} \\
	&\geq \frac{1}{t}\left\langle \nabla h\big(G(\bm{x}')\big), G(\bm{x}'+t\bm{v}) - G(\bm{x}') \right\rangle -\frac{L_h}{2t} \left\| G(\bm{x}'+t\bm{v}) - G(\bm{x}') \right\|^2 \\
	&\geq \frac{1}{t}\left\langle \nabla h\big(G(\bm{x})\big), G(\bm{x}'+t\bm{v}) - G(\bm{x}') \right\rangle  -\frac{L_h L_g^2}{2}\left\| \bm{v} \right\|^2 \cdot t 
	 -L_hL_g^2 \|\bm{v}\|\cdot \|\bm{x}-\bm{x}'\| \\
	 &= \Delta_t \bar{f}(\bm{x}')(\bm{v})  -\frac{L_h L_g^2}{2}\left\| \bm{v} \right\|^2 \cdot t 
	 -L_hL_g^2 \|\bm{v}\|\cdot \|\bm{x}-\bm{x}'\| 
\end{align*}
Therefore, for any $\bm{v}\in\mathbb{R}^d$, we know 
\[
\begin{aligned}
f^\circ(\bm{x};\bm{v}) &= \lim_{\epsilon\searrow 0} \sup_{\|\bm{x}'- \bm{x}\|\leq \epsilon} \sup_{0 < t < \epsilon} \Delta_t f(\bm{x}')(\bm{v}), \\
&\overset{(i)}{\geq} \lim_{\epsilon\searrow 0} \sup_{\|\bm{x}'- \bm{x}\|\leq \epsilon} \sup_{0 < t < \epsilon} \Delta_t \bar{f}(\bm{x}')(\bm{v})  - \left(\lim_{\epsilon\searrow 0} \frac{L_h L_g^2}{2}\left\| \bm{v} \right\|^2 \cdot \epsilon\right)
	 -\left(\lim_{\epsilon\searrow 0} L_hL_g^2 \|\bm{v}\|\cdot \epsilon\right) \\
&=\lim_{\epsilon\searrow 0} \sup_{\|\bm{x}'- \bm{x}\|\leq \epsilon} \sup_{0 < t < \epsilon} \Delta_t \bar{f}(\bm{x}')(\bm{v}) \\
&=\bar{f}^\circ(\bm{x};\bm{v}),
\end{aligned}
\]
where in $(i)$ we use $\sup f-g \geq \sup f - \sup g$.
For the converse direction $f^\circ(\bm{x};\bm{v}) \leq \bar{f}^\circ(\bm{x};\bm{v})$, we just compute similarly.
We have proved $f^\circ(\bm{x};\bm{v}) = \bar{f}^\circ(\bm{x};\bm{v}), \forall \bm{v}\in\mathbb{R}^d$. The claim follows from the correspondence between sublinear $f^\circ$ and convex $\partial_C f$ \cite[Proposition 2.1.5]{clarke1990optimization}.

Now we show the relation holds for Fr\'echet subdifferential. As $h$ and $G$ are locally Lipschitz and directional differentiable, they are Bouligand-differentiable (B-differentiable) according to \cite[Definition 4.1.1]{cui2021modern}. Then, by \cite[Proposition 4.1.2(a)]{cui2021modern}, we know that
\[
f'(\bm{x};\bm{d}) = h'\left( G(\bm{x}); G'(\bm{x};\bm{d}) \right) = \left\langle \nabla h\big(G(\bm{x})\big), G'(\bm{x};d) \right\rangle,
\]
where the directional derivative $G'(\bm{x};\bm{d})$ is defined element-wise as $\left( G'_i(\bm{x}; \bm{v}) \right)_{i=1}^n$ according to \cite[Definition 1.1.4]{cui2021modern}. Thus, combined with $\bar{f}'(\bm{x};\bm{d}) = \left\langle \nabla h\big(G(\bm{x})\big), G'(\bm{x};d) \right\rangle$, we have shown $\bar{f}'(\bm{x};\bm{d}) = f'(\bm{x};\bm{d})$ for any $\bm{d}$, which implies 
\[
\widehat{\partial} f(\bm{x}) = \widehat{\partial} \Big[\left\langle \nabla h\big(G(\bm{x})\big), G(\cdot ) \right\rangle \Big](\bm{x})
\]
by \cite[Exercise 8.4]{rockafellar2009variational} (note that for B-differentiable $f$, the subderivative $df(\bm{x})(\bm{d})$ in \cite[Exercise 8.4]{rockafellar2009variational} is equal to the directional derivative $f'(\bm{x};\bm{d})$  by \cite[Exercise 9.15]{rockafellar2009variational}).
\end{proof}

\begin{Remark}
\Cref{thm:linearization} is fundamentally different from the classic exact chain rule as the exact chain rule does not hold even for very simple function. Consider $h(a,b) = a-b$ and $G(x) = (|x|, |x|)$. We have $\partial_C [h\circ G](0) = \{0\} \subsetneq [-1,1]+[-1,1]=[-2,2]$. In contrast, by \Cref{thm:linearization}, we have $\partial_C [h\circ G](0) = \partial_C [|\cdot|-|\cdot|](0)=\{0\}$. One should compare \Cref{thm:linearization} with \cite[Theorem 2.3.9, Theorem 2.3.10]{clarke1990optimization}. Besides, \Cref{thm:linearization} implies \cite[Theorem 2.3.9(ii)]{clarke1990optimization}.
\end{Remark}

\begin{Corollary}\label{coro:prod}
	Let $f:\mathbb{R}\times \mathbb{R}^d \rightarrow \mathbb{R}$ be $f(u,\bm{x})=u\cdot g(\bm{x})$, where $g:\mathbb{R}^d\rightarrow\mathbb{R}$ is a Lipschitz function. Then, we have $\partial_\triangleleft f(u,\bm{x}) = \{g(\bm{x})\}\times \partial_\triangleleft [u\cdot g](\bm{x})$.
\end{Corollary}
\begin{proof}
	Let $h:\mathbb{R}\times \mathbb{R}\rightarrow \mathbb{R}$ be $h(a,b)=a \cdot b$. It is easy to see $h$ is smooth at any $(a,b)$. Let $C(u, \bm{x}) = (u, g(\bm{x})), \bar{u}=u$ and $\bar{\bm{x}}=\bm{x}$. As $f(u,\bm{x}) = h\circ C(u,\bm{x})$, by \Cref{thm:linearization}, we know
	\[
	\partial_\triangleleft f(u,\bm{x}) = \partial_\triangleleft \big[g(\bar{\bm{x}})\cdot u + \bar{u}\cdot g(\bm{x})\big](u,\bm{x}) = \{g(\bm{x})\}\times \partial_\triangleleft [u\cdot g](\bm{x}),
	\]
	as required.
\end{proof}

\subsection{Exact Chain Rule of a Non-Separable Piecewise Linear Function}\label{sec:prf-CR-pl}
In this section, we consider the validity of the exact subdifferential chain rule of a simple piecewise-linear function, which is defined by
\[
f_{\textsf{PL}}(\bm{w})\coloneqq \sum_{i=1}^n \max\lB \bm{x}_i^\top \bm{w},0 \rB - \sum_{j=1}^m \max\lB \bm{y}_j^\top \bm{w}, 0 \rB.
\]

\subsubsection{Chain Rule for Clarke Subdifferential}
\begin{Theorem}[Clarke]\label{thm:pl-CR}
	Suppose $\bm{x}_i^\top \bm{w} = \bm{y}_j^\top \bm{w} = 0$ for any $i\in [n], j \in [m]$. We have the exact Clarke subdifferential chain rule
	\[
	\partial_C f_{\textsf{PL}} (\bm{w}) = G_{\textsf{PL}}^C \coloneqq \sum_{i=1}^n \bm{x}_i\cdot [0,1] + \sum_{j=1}^m (-\bm{y}_j) \cdot[0,1]
	\]
	if and only if
	$
	\spn \big(\lB \bm{x}_i \rB_{i=1}^n \big) \cap 
	\spn \big(\lB \bm{y}_j \rB_{j=1}^m \big) = \{\bm{0}\}.
	$
\end{Theorem}
\begin{proof}
We have divided the proof into \Cref{lem:s-nece} and \Cref{lem:s-suff}.
\end{proof}

\begin{Lemma}[Necessary]\label{lem:s-nece}
	If there exists $\bm{v}\in\mathbb{R}^d$ such that
	\[
	\bm{0}\neq \bm{v}\in\spn \big(\lB \bm{x}_i \rB_{i=1}^n \big) \cap 
	\spn \big(\lB \bm{y}_j \rB_{j=1}^m \big),
	\]
	then $
	\partial_C f_{\textsf{PL}} (\bm{w}) \subsetneq G_{\textsf{PL}}^C.
	$
\end{Lemma}
\begin{proof}
	We first prove that assuming certain regularity on $\{\bm{x}_i\}_{i=1}^n$ and $\{\bm{y}_j\}_{j=1}^{m}$ is without loss of generality. Let the indices set $\mathcal{J}_x \subseteq [n]$  be a selection from $\{\bm{x}_i\}_{i=1}^n$ such that $\lB \bm{x}_i \rB_{i\in\mathcal{J}_x}$ are linearly independent and satisfy
	\[
	\spn \Big(\lB \bm{x}_i \rB_{i\in\mathcal{J}_x} \Big) = \spn \Big(\lB \bm{x}_i \rB_{i=1}^n \Big).
	\]
	Similarly, we define $\mathcal{J}_y \subseteq [m]$  for $\{\bm{y}_j\}_{j=1}^m$.
	Then, we write
	\[
f_{\textsf{PL}}(\bm{w})= f_{\textsf{PL}1}(\bm{w}) + f_{\textsf{PL}2}(\bm{w}),
\]
where
\[
\begin{aligned}
f_{\textsf{PL}1}(\bm{w})&\coloneqq\sum_{i\in[n]\backslash\mathcal{J}_x} \max\lB \bm{x}_i^\top \bm{w},0 \rB - \sum_{j\in[m]\backslash\mathcal{J}_y } \max\lB \bm{y}_j^\top \bm{w}, 0 \rB,\\
f_{\textsf{PL}2}(\bm{w})&\coloneqq\sum_{i\in\mathcal{J}_x} \max\lB \bm{x}_i^\top \bm{w},0 \rB - \sum_{j\in\mathcal{J}_y } \max\lB \bm{y}_j^\top \bm{w}, 0 \rB.	
\end{aligned}
\]
By the fuzzy sum rule \cite[Proposition 2.3.3]{clarke1990optimization}, we know
\[
\partial_C f_{\textsf{PL}}(\bm{w}) \subseteq \partial_C f_{\textsf{PL}1}(\bm{w}) + \partial_C f_{\textsf{PL}2}(\bm{w}) \subseteq \partial_C f_{\textsf{PL}1}(\bm{w}) + G_{\textsf{PL}2}^C \subseteq G_{\textsf{PL}}^C,
\]
where we define $G_{\textsf{PL}2}^C\coloneqq \sum_{i\in\mathcal{J}_x} \bm{x}_i\cdot [0,1] + \sum_{j\in\mathcal{J}_y} (-\bm{y}_j) \cdot[0,1]$.
Thus, to prove $\partial_C f_{\textsf{PL}}(\bm{w}) \subsetneq G_{\textsf{PL}}^C$, by \Cref{lem:localization} and \cite[Proposition 2.1.2(a)]{clarke1990optimization}, we only need to show $\partial_C f_{\textsf{PL}2}(\bm{w}) \subsetneq G_{\textsf{PL}2}^C$. So, by abuse of notation and focus on $f_{\textsf{PL}2}$, we assume $\{\bm{x}_i\}_{i=1}^n$ are linearly independent. Similarly, we assume $\{\bm{y}_j\}_{j=1}^m$ are linearly independent. In the following, we may use \Cref{lem:localization} and above argument implicitly to assume regularity for simplicity. As $\bm{v}\in\spn \big(\lB \bm{x}_i \rB_{i=1}^n \big) \cap 
	\spn \big(\lB \bm{y}_j \rB_{j=1}^m \big)$, we write
	\begin{equation}\label{eq:prf-def-of-v}
			\bm{0}\neq \bm{v} = \sum_{i=1}^n a_i \cdot \bm{x}_i = \sum_{j=1}^m b_i \cdot \bm{y}_j.
	\end{equation}
	It is safe to assume $a_i\neq 0, b_j \neq 0, \forall i \in [n],j\in[m]$. Fix $\bm{y}_m$. We can further assume $\{\bm{x}_i\}_{i=1}^n \cup \{\bm{y}_j\}_{j=1}^{m-1}$ are linearly independent. To see this, suppose to the contrary $\{\bm{x}_i\}_{i=1}^n \cup \{\bm{y}_j\}_{j=1}^{m-1}$ are not linearly independent, we get
	$
	\bm{0} = \sum_{i=1}^n p_i\cdot \bm{x}_i + \sum_{j=1}^{m-1}q_j\cdot \bm{y}_j.
	$ 
	We know that there exist $j' \in [m-1]$ such that $q_{j'}\neq 0$, as otherwise by linear independence of $\{\bm{x}_i\}_{i=1}^m$, for any $i\in[n]$, it holds that $p_i=0$, hence that $\{\bm{x}_i\}_{i=1}^n \cup \{\bm{y}_j\}_{j=1}^{m-1}$ are linearly independent.
	As $q_{j'}\neq 0$, we have 
	\[
	\bm{y}_{j'} = -\sum_{i=1}^n (p_i/q_{j'})\cdot \bm{x}_i-\sum_{j\in [m-1]\backslash\{j'\}} (q_j/q_{j'})\cdot \bm{y}_j.
	\]
	Plug in to \Cref{eq:prf-def-of-v} and $\bm{y}_{j'}$ is removed. Repeat this procedure and by abuse of notation, we have $\{\bm{x}_i\}_{i=1}^n \cup \{\bm{y}_j\}_{j=1}^{m-1}$ are linearly independent. After that, we exam $\{a_i\}_i$ and $\{b_j\}_j$. We remove $\bm{x}_i$ if $a_i=0$ and remove $\bm{y}_j$ if $b_j=0$, which is without of generality by \Cref{lem:localization}.	
	It is possible that all $\{\bm{y}_j\}_{j=1}^{m-1}$ are removed and we get $m=1$ and $\bm{y}_m \in \spn\big(\{\bm{x}_i\}_{i=1}^n\big)$. But as $\bm{y}_m \neq \bm{0}$, we always have $n\geq 1$. Then, we can write
	\begin{equation}\label{eq:ym}
			\bm{y}_m = \sum_{i=1}^n \alpha_i \bm{x}_i + \sum_{j=1}^{m-1} \beta_j \bm{y}_j,
	\end{equation}
	with $\alpha_i \neq 0, \beta_j \neq 0$ for any $i \in [n], j \in [m-1]$. Note that, for such $\{\bm{x}_i\}_{i=1}^n$ and $\{\bm{y}_j\}_{j=1}^{m-1}$, we have the exact chain rule
		\[
	\partial_C \left[ \sum_{i=1}^n \max\lB \bm{x}_i^\top \cdot,0 \rB - \sum_{j=1}^{m-1} \max\lB \bm{y}_j^\top \cdot, 0 \rB\right](\bm{w}) = \sum_{i=1}^n \bm{x}_i\cdot [0,1] + \sum_{j=1}^{m-1} (-\bm{y}_j) \cdot[0,1]
	\]
	by using \cite[Theorem 2.3.10]{clarke1990optimization} and linear independence. We proceed to show that $\partial_C f_{\textsf{PL}}(\bm{w}) \subsetneq G_{\textsf{PL}}^C$ by exhibiting an element in $G_{\textsf{PL}}^C \backslash \partial_C f_{\textsf{PL}}(\bm{w})$. Let $\bm{\theta}\in\mathbb{R}_+^{n+m}$ and we define
	\[\theta_i =     \left\{ \begin{array}{rcl}
         |\alpha_i| & \mbox{for}
         & 1\leq i \leq n \\ |\beta_{i-n}|  & \mbox{for} & n+1 \leq i \leq n+m-1, \\
         1 & \mbox{for} & i = m+n
                \end{array}\right. \quad
	\bm{A} = \left[ \begin{array}{c}
     \sgn(\alpha_1)\cdot \bm{x}_1^\top \\
     \vdots \\
     \sgn(\alpha_n)\cdot \bm{x}_n^\top \\
      \sgn(\beta_1)\cdot \bm{y}_1^\top \\
     \vdots \\
     \sgn(\beta_{m-1})\cdot \bm{y}_{m-1}^\top \\
     -\bm{y}_m^\top
     \end{array} \right]
     \in \mathbb{R}^{(n+m)\times d}.
     	\]
	Note that 
	\[
	\bm{A}^\top \bm{\theta} = \sum_{i=1}^n \alpha_i \bm{x}_i + \sum_{j=1}^{m-1} \beta_j \bm{y}_j - \bm{y}_m = \bm{0}.
	\]
	By Gordan's Theorem in \Cref{lem:gordan}, we have certified the nonexistence of direction $\bm{d}\in\mathbb{R}^d$ such that
	\begin{equation}\label{eq:sgn-gordan-1}
	\left\{ \begin{array}{rcl}
         \sgn(\alpha_i)\cdot \bm{d}^\top \bm{x}_i < 0 & \mbox{for}
         & i \in [n] \\
         \sgn(\beta_j)\cdot \bm{d}^\top \bm{y}_j < 0 & \mbox{for}
         & j \in [m-1]\\
          \bm{d}^\top \bm{y}_m > 0 &
         & 
             \end{array}\right..	
	\end{equation}
	By $-\bm{A}^\top \bm{\theta} = \bm{0}$, similarly, we certify the nonexistence of direction $\bm{d}\in\mathbb{R}^d$ such that
	\begin{equation}\label{eq:sgn-gordan-2}
	\left\{ \begin{array}{rcl}
         \sgn(\alpha_i)\cdot \bm{d}^\top \bm{x}_i > 0 & \mbox{for}
         & i \in [n] \\
         \sgn(\beta_j)\cdot \bm{d}^\top \bm{y}_j > 0 & \mbox{for}
         & j \in [m-1]\\
          \bm{d}^\top \bm{y}_m < 0 &
         & 
             \end{array}\right..
	\end{equation}
	Let the Bouligand subdifferential of $f_{\textsf{PL}}$ at $\bm{w}$ be $\partial_B f_{\textsf{PL}}(\bm{w})$; see \cite[Definition 4.3.1]{cui2021modern}. Define
	\begin{align*}
	\bm{\nabla}_1 &\coloneqq \sum_{i=1}^m \mathbf{1}_{\alpha_i > 0} \cdot \bm{x}_i - \sum_{j=1}^{m-1}\mathbf{1}_{\beta_j > 0}\cdot \bm{y}_j, \tag{compare \eqref{eq:sgn-gordan-1}} \\
	\bm{\nabla}_2 &\coloneqq \sum_{i=1}^m \mathbf{1}_{\alpha_i < 0} \cdot \bm{x}_i - \sum_{j=1}^{m-1}\mathbf{1}_{\beta_j < 0}\cdot \bm{y}_j - \bm{y}_m. \tag{compare \eqref{eq:sgn-gordan-2}} 
	\end{align*}
	By \cite[Proposition 4.4.8(c)]{cui2021modern} and the nonexistences of $\bm{d}$ for \eqref{eq:sgn-gordan-1} and \eqref{eq:sgn-gordan-2}, we have proved that $\bm{\nabla}_1, \bm{\nabla}_2 \notin \partial_B f_{\textsf{PL}}(\bm{w})$. Let us define a set
	\[
	\overline{G}_{\textsf{PL}}^C \coloneqq \sum_{i=1}^n \bm{x}_i \cdot \{0,1\} + \sum_{j=1}^m (-\bm{y}_j)\cdot \{0,1\} \subseteq \mathbb{R}^d.
	\]
	Besides, using \cite[Proposition 4.4.8(c)]{cui2021modern}, we have $\partial_B f_{\textsf{PL}}(\bm{w}) \subseteq \overline{G}_{\textsf{PL}}^C\backslash\{\bm{\nabla}_1, \bm{\nabla}_2\}$.
	Then, with \cite[Theorem 9.61]{rockafellar2009variational}, it follows that
	\[
	\partial_C f_{\textsf{PL}}(\bm{w}) = \conv(\partial_B f_{\textsf{PL}}(\bm{w})) \subseteq \conv(\overline{G}_{\textsf{PL}}^C\backslash\{\bm{\nabla}_1, \bm{\nabla}_2\}).
	\]
	 Therefore, to prove $\partial_C f_{\textsf{PL}}(\bm{w}) \subsetneq G_{\textsf{PL}}^C$, we only need to show
	\[
	\bm{\nabla}_1 \in G_{\textsf{PL}}^C\backslash \conv\left(\overline{G}_{\textsf{PL}}^C\backslash\{\bm{\nabla}_1, \bm{\nabla}_2\}\right).
	\]
	To this end, we define two sets satisfying  $\overline{G}_{\textsf{PL}}^C = P_1 \cup P_2$ as
	\[
	\begin{aligned}
		P_1 &\coloneqq \sum_{i=1}^n \bm{x}_i \cdot \{0,1\} + \sum_{j=1}^{m-1} (-\bm{y}_j)\cdot \{0,1\}, \\
		P_2 &\coloneqq \sum_{i=1}^n \bm{x}_i \cdot \{0,1\} + \sum_{j=1}^{m-1} (-\bm{y}_j)\cdot \{0,1\} - \bm{y}_m.
	\end{aligned}
	\]
	Thus, we can write $\overline{G}_{\textsf{PL}}^C\backslash\{\bm{\nabla}_1, \bm{\nabla}_2\} \subseteq \left(P_1 \backslash \{\bm{\nabla}_1\} \right)\cup\left( P_2 \backslash \{\bm{\nabla}_2\}  \right)$. 
	It is evident that $\bm{\nabla}_1 \in G_{\textsf{PL}}$.
	If $\bm{\nabla}_1 \in \conv\left(\overline{G}_{\textsf{PL}}^C\backslash\{\bm{\nabla}_1, \bm{\nabla}_2\}\right)$, we have
	\[
	\bm{\nabla}_1 = \lambda \bm{g}^{P_1} + (1-\lambda)\bm{g}^{P_2}, \quad\textnormal{with}\quad 
	 \bm{g}^{P_1} \in \conv\left( P_1 \backslash \{\bm{\nabla}_1\} \right),
	 \bm{g}^{P_2} \in \conv\left( P_2 \backslash \{\bm{\nabla}_2\} \right).
	\]
	We now show that it must be $\lambda = 1$ by considering three cases:
	
	\paragraph{Case 1.} $\exists i \in [n]: \alpha_i > 0$. Without loss of generality, we assume $\alpha_1 > 0$. Note that for any $\bm{g}^{P_2} \in \conv\left( P_2 \backslash \{\bm{\nabla}_2\} \right)$, using the representation of $\bm{y}_m$ in \Cref{eq:ym}, we have 
	\[
	\bm{g}^{P_2} = \sum_{i=1}^n (\gamma_i - \alpha_i)\cdot \bm{x}_i + \sum_{j=1}^{m-1} (\gamma_{n+j}+\beta_j)\cdot(-\bm{y}_j),
	\]
	where $\gamma_k \in [0,1], \forall k \in [n+m-1]$. Similarly, we write $\bm{g}^{P_1} \in \conv\left( P_1 \backslash \{\bm{\nabla}_1\} \right)$ as
	\[
	\bm{g}^{P_1} = \sum_{i=1}^n \mu_i\cdot \bm{x}_i + \sum_{j=1}^{m-1} \mu_{n+j}\cdot(-\bm{y}_j),
	\]
	where $\mu_k \in [0,1], \forall k \in [n+m-1]$. Therefore, we know
	\begin{align*}
	\bm{\nabla}_1&= \lambda \bm{g}^{P_1} + (1-\lambda)\bm{g}^{P_2} \\
	&= 
	\sum_{i=1}^n \big(\lambda\cdot\mu_i+(1-\lambda)\cdot(\gamma_i - \alpha_i)\big)\cdot \bm{x}_i + \sum_{j=1}^{m-1} \big(\lambda\cdot\mu_{n+j}+(1-\lambda)\cdot(\gamma_{n+j}+\beta_j)\big)\cdot(-\bm{y}_j) \\
	&= \sum_{i=1}^m \mathbf{1}_{\alpha_i > 0} \cdot \bm{x}_i - \sum_{j=1}^{m-1}\mathbf{1}_{\beta_j > 0}\cdot \bm{y}_j. \tag{by the definition of $\bm{\nabla}_1$}
	\end{align*}
	As $\{\bm{x}_i\}_{i=1}^n \cup \{\bm{y}_j\}_{j=1}^{m-1}$ are linearly independent, it holds 
	\[
	\lambda\cdot\mu_1+(1-\lambda)\cdot(\gamma_1 - \alpha_1)=\mathbf{1}_{\alpha_1 > 0} = 1.
	\]
	If $0\leq \lambda < 1$, we have
	\[
	1=\lambda\cdot\mu_1+(1-\lambda)\cdot(\gamma_1 - \alpha_1) \leq 1-(1-\lambda)\cdot \alpha_1 < 1,
	\]
	which gives the contradiction.
	\paragraph{Case 2.} $\forall i \in [n]: \alpha_i < 0$ but $\exists j \in [m-1]: \beta_j > 0$. Suppose $\beta_1 > 0$. 
	Then, we write
	\[
	\bm{y}_1 = \sum_{i=1}^n (-\alpha_i/\beta_1)\cdot \bm{x}_1 + \sum_{j=2}^{m-1} (-\beta_j/\beta_1)\cdot \bm{y}_j - (1/\beta_1)\cdot \bm{y}_m.
	\]
	Note that $\{\bm{x}_i\}_{i=1}^n \cup \{\bm{y}_j\}_{j=1}^{m}$ are linearly independent.
	By abuse of notation and swapping $\bm{y}_m$ and $\bm{y}_1$, we still write $	\bm{y}_m = \sum_{i=1}^n \alpha_i \bm{x}_i + \sum_{j=1}^{m-1} \beta_j \bm{y}_j
$.  Then, we have $\forall i \in [n]: \alpha_i > 0$ and the situation reduces to the \textbf{Case 1}.
	
	\paragraph{Case 3.} $\forall i \in [n], j \in [m-1]: \alpha_i < 0, \beta_j < 0$. In that case, we have $\bm{\nabla}_1 = \bm{0}$. By a similar manipulation as these in \textbf{Case 1}, we have
	\[
	\begin{aligned}
	\bm{\nabla}_1&=  
	\sum_{i=1}^n \big(\lambda\cdot\mu_i+(1-\lambda)\cdot(\gamma_i - \alpha_i)\big)\cdot \bm{x}_i + \sum_{j=1}^{m-1} \big(\lambda\cdot\mu_{n+j}+(1-\lambda)\cdot(\gamma_{n+j}+\beta_j)\big)\cdot(-\bm{y}_j) \\
	&= \bm{0}.
	\end{aligned}
	\]
	As $\{\bm{x}_i\}_{i=1}^n \cup \{\bm{y}_j\}_{j=1}^{m-1}$ are linearly independent, it holds 
	\[
	\lambda\cdot\mu_1+(1-\lambda)\cdot(\gamma_1 - \alpha_1)=\mathbf{1}_{\alpha_1 > 0} = 0.
	\]
	If $0\leq \lambda < 1$, we have
	\[
	0=\lambda\cdot\mu_1+(1-\lambda)\cdot(\gamma_1 - \alpha_1) \geq -(1-\lambda)\cdot \alpha_1 > 0,
	\]
	which gives the contradiction.
	
	Therefore, we have shown  $\lambda=1$ which implies $
	\bm{\nabla}_1  \in \conv\left( P_1 \backslash \{\bm{\nabla}_1\} \right).
	$ However, as $\{\bm{x}_i\}_{i=1}^n \cup \{\bm{y}_j\}_{j=1}^{m-1}$ are linearly independent, $\bm{\nabla}_1$ is an extreme point of $\conv (P_1)$ by \Cref{lem:extremept}. Thus, we know
	$
	\bm{\nabla}_1  \notin \conv\left( P_1 \backslash \{\bm{\nabla}_1\} \right)
	$ by definition,
	a contradiction.
\end{proof}

\begin{Lemma}[Sufficient]\label{lem:s-suff}
	If the following condition holds
	\[
	\spn \big(\lB \bm{x}_i \rB_{i=1}^n \big) \cap 
	\spn \big(\lB \bm{y}_j \rB_{j=1}^m \big) = \{\bm{0}\},
	\]
	then $
	\partial_C f_{\textsf{PL}} (\bm{w}) = G_{\textsf{PL}}^C.
	$
\end{Lemma}
\begin{proof}
We first do a general preparation that will be reused in other developments.
Let $\bm{X} =  \left[ \begin{array}{c|c|c}
		\bm{x}_1 &\cdots &\bm{x}_n \end{array} \right] \in \mathbb{R}^{d\times n}$ and $\bm{Y} = \left[ \begin{array}{c|c|c}
		\bm{y}_1 &\cdots &\bm{y}_m \end{array} \right] \in \mathbb{R}^{d\times m}$ be given. 
The thin-SVD of $\bm{X}$ can be written as $\bm{X} = \bm{U}_x \bm{\Sigma}_x \bm{V}_x^\top$ with $\bm{U}_x \in \St(d, r_x), \bm{\Sigma}_x \in \mathbb{R}^{r_x \times r_x}, \bm{V}_x \in \St(n, r_x)$, and $r_x = \rk(\bm{X})$. Similarly, for $\bm{Y}$, we have $\bm{Y} = \bm{U}_y \bm{\Sigma}_y \bm{V}_y^\top$ with $\bm{U}_y \in \St(d, r_y), \bm{\Sigma}_y \in \mathbb{R}^{r_y \times r_y}, \bm{V}_y \in \St(m, r_y)$, and $r_y = \rk(\bm{Y})$. As $\spn \big(\lB \bm{x}_i \rB_{i=1}^n \big) \cap 
	\spn \big(\lB \bm{y}_j \rB_{j=1}^m \big) = \{\bm{0}\}$, we know $\bm{U}_x^\top \bm{U}_y = \bm{0}$. Therefore, we can write
	\[
	\mathbb{R}^{n+m}\ni\left[ \begin{array}{c}
     \bm{X}^\top \\ \bm{Y}^\top \end{array} \right] \cdot \bm{w} = 
     \left[ \begin{array}{cc}
     \bm{V}_x \bm{\Sigma}_x & \\
     & \bm{V}_y \bm{\Sigma}_y \end{array} \right] \cdot \left( \bm{z} \coloneqq
     \left[ \begin{array}{c}
     \bm{z}_1 \\ \bm{z}_2 \end{array} \right]\in\mathbb{R}^{r_x+r_y}\right),
	\]
	where $\bm{z} = \bm{U}^\top \bm{w}$ and $\bm{U}\coloneqq \left[ \begin{array}{c|c}
     \bm{U}_x & \bm{U}_y \end{array} \right] \in \St(d, r_x+r_y).$
     
	Let an auxiliary function $h_{\textsf{PL}}:\mathbb{R}^{r_x}\times \mathbb{R}^{r_y}\rightarrow \mathbb{R}$ be
	\[
			h_{\textsf{PL}}(\bm{z}_1, \bm{z}_2)\coloneqq 
	\underbrace{\sum_{i=1}^n \max\lB \bm{e}_i^\top\bm{V}_{x} \bm{\Sigma}_x \bm{z}_1,0 \rB}_{h_{\textsf{PL}1}(\bm{z}_1)} - \underbrace{\sum_{j=1}^m \max\lB \bm{e}_j^\top\bm{V}_{y} \bm{\Sigma}_y \bm{z}_2, 0 \rB}_{h_{\textsf{PL}2}(\bm{z}_2)}. 
	\]
	As $h_{\textsf{PL}}$ is separable with respect to $\bm{z}_1$ and $\bm{z}_2$, by \cite[Proposition 2.5]{rockafellar1985extensions} and \citep[Proposition 10.5]{rockafellar2009variational}, we know 
	\[
	\partial_\triangleleft h_{\textsf{PL}}(\bm{z_1},\bm{z_2}) = \partial_\triangleleft h_{\textsf{PL}1}(\bm{z}_1)\times  \partial_\triangleleft [-h_{\textsf{PL}2}](\bm{z}_2).
	\]
	Note that $f_{\textsf{PL}}(\bm{w}) = h_{\textsf{PL}}(\bm{U}_x^\top \bm{w},\bm{U}_y^\top \bm{w})$.
	We compute
	\begin{align*}
	\partial_\triangleleft f_{\textsf{PL}}(\bm{w}) &= \partial_\triangleleft \Big[ h_{\textsf{PL}}(\bm{U}_x^\top \cdot,\bm{U}_y^\top \cdot)\Big](\bm{w}) \\
	&\overset{(a)}{=} \bm{U}\partial_\triangleleft \big[ h_{\textsf{PL}}( \cdot, \cdot)\big]\left(\bm{U}_x^\top\bm{w}, \bm{U}_y^\top\bm{w}\right)  \\
	&\overset{(b)}{=} \bm{U}_x\partial_\triangleleft \big[h_{\textsf{PL}1}\big]\left(\bm{U}_x^\top\bm{w}\right)+ \bm{U}_y\partial_\triangleleft [-h_{\textsf{PL}2}]\left(\bm{U}_y^\top\bm{w}\right) \\%\tag{by \cite[Proposition 10.5]{rockafellar2009variational}} \\
	&\overset{(c)}{=} \partial_\triangleleft \left[h_{\textsf{PL}1}\left(\bm{U}_x^\top \cdot \right) \right](\bm{w}) + \partial_\triangleleft  \left[-h_{\textsf{PL}2}\left(\bm{U}_y^\top \cdot \right) \right](\bm{w}), \tag{$\diamondsuit$}\label{eq:sufficient-prep-end} %
	\end{align*}
	where (a) is using \cite[Theorem 2.3.10]{clarke1990optimization}, 
	\cite[Exercise 10.7]{rockafellar2009variational}, and $\bm{U}$ is full column rank; (b) is from $\partial_\triangleleft h_{\textsf{PL}}(\bm{z_1},\bm{z_2}) = \partial_\triangleleft h_{\textsf{PL}1}(\bm{z}_1)\times  \partial_\triangleleft [-h_{\textsf{PL}2}](\bm{z}_2)$; (c) is using the reasoning in (a) for $h_1,h_2$ separately.
	
	In particular for Clarke subdifferential, we know $\partial_C [-h_{\textsf{PL}2}] = -\partial_C [h_{\textsf{PL}2}]$ using \cite[Proposition 2.3.1]{clarke1990optimization}. As $h_{\textsf{PL}2}$ is convex, $\partial_C [h_{\textsf{PL}2}]$ is equal to the convex subdifferential of $h_{\textsf{PL}2}$ by \cite[Proposition 2.2.7]{clarke1990optimization}. Then, by \cite[\S D, Corollary 4.3.2]{hiriart2004fundamentals}, a direct computation gives
	\[
	\partial_C f_{\textsf{PL}}(\bm{w}) = \partial_C \left[h_{\textsf{PL}1}\left(\bm{U}_x^\top \cdot \right) \right](\bm{w}) +\left(- \partial_C  \left[h_{\textsf{PL}2}\left(\bm{U}_y^\top \cdot \right) \right](\bm{w})\right) = G_{\textsf{PL}}^C,
	\]
	as required.
\end{proof}

\begin{proof}[Proof of \Cref{thm:reluMain-CR}]
	According to the argument in \Cref{sec:roadmap}, we only need to consider the Clarke subdifferential $\partial_C f_k(\bm{w}_k)$ for every $k \in [H]$. It is showed in \Cref{thm:pl-CR} that we have
\[
\partial_C f_k(\bm{w}_k) = \widetilde{G}_k^C,
\]
if and only if the following span qualification is satisfied:
\[
\spn \left(\lB \bm{x}_i \rB_{i\in\mathcal{I}_k^+} \right) \cap 
	\spn \left(\lB \bm{x}_j \rB_{j\in\mathcal{I}_k^-} \right) = \{\bm{0}\}.
\]
Then, put all $k \in [H]$ cases together, and \Cref{thm:reluMain-CR} is proved.
\end{proof}

\subsubsection{Chain Rule for Limiting Subdifferential}
\begin{Theorem}[Limiting]\label{thm:pl-CR-limiting}
	Suppose $\bm{x}_i^\top \bm{w} = \bm{y}_j^\top \bm{w} = 0$ and $\bm{y}_j\neq \bm{0}$ for any $i\in [n], j \in [m]$. We have the exact limiting subdifferential chain rule
	\[
	\partial f_{\textsf{PL}} (\bm{w}) = G_{\textsf{PL}}^L \coloneqq \sum_{i=1}^n \bm{x}_i\cdot [0,1] + \left\{ -\sum_{j=1}^m \bm{y}_j \cdot \mathbf{1}_{\bm{d}^\top \bm{y}_j>0}: \bm{d}\in\mathbb{R}^d, \min_{1 \leq t \leq m} \left| \bm{d}^\top \bm{y}_t \right| > 0  \right\}
	\]
	if and only if
	$
	\spn \big(\lB \bm{x}_i \rB_{i=1}^n \big) \cap 
	\spn \big(\lB \bm{y}_j \rB_{j=1}^m \big) = \{\bm{0}\}.
	$
\end{Theorem}
\begin{proof}
	(Sufficient) We begin with the general argument in the proof of \Cref{lem:s-suff} until \Cref{eq:sufficient-prep-end}. After that, we will focus on the proof of 
	\[
	\partial  \left[-h_{\textsf{PL}2}\left(\bm{U}_y^\top \cdot \right) \right](\bm{w}) = \left\{ -\sum_{j=1}^m \bm{y}_j \cdot \mathbf{1}_{\bm{d}^\top \bm{y}_j>0}: \bm{d}\in\mathbb{R}^d, \min_{1 \leq t \leq m} \left| \bm{d}^\top \bm{y}_t \right| > 0  \right\}\eqqcolon G_{\textsf{PL}2}^L.
	\]
	For the ease of notation, we denote $q(\bm{w}) \coloneqq -h_{\textsf{PL}2}\left(\bm{U}_y^\top \bm{w} \right) = -\sum_{j=1}^m \max\{\bm{y}_j^\top\bm{w},0\}$. Note that by the definition of limiting subdifferential (see \Cref{def:subd-l}), we have 
	\[
	\partial q(\bm{w}) = \limsup_{\bm{w}' \rightarrow \bm{w}} \widehat{\partial} q(\bm{w}') = \left\{ \bm{g}: \exists \{\bm{w}_\nu\} \rightarrow \bm{w} \textnormal{ and } \{\bm{g}_\nu\} \rightarrow \bm{g} \textnormal{ s.t. } \bm{g}_\nu \in \widehat{\partial} q(\bm{w}_\nu), 
	\forall \nu \right\}.
	\]
	Let $\bm{g} \in \partial q(\bm{w})$. Then, there exist $\{\bm{w}_\nu\}_\nu$ and $\{\bm{g}_\nu\}_\nu$ such that $\bm{w}_\nu \rightarrow \bm{w}, \bm{g}_\nu \in \widehat{\partial} q(\bm{w}_\nu),$ and  $\bm{g}_\nu \rightarrow \bm{g}$. We can assume for any $\nu$ and any $j\in[m]$, we have $\bm{w}_\nu^\top \bm{y}_j \neq 0$, as otherwise, by \Cref{lem:frechet-empty}, $\widehat{\partial} q(\bm{w}_k) = \emptyset$ and $\bm{g}_k$ is undefined. Then, for any $\nu$, the function $q$ is strictly differentiable at $\bm{w}_\nu$, which implies
	\[
	\{\bm{g}_\nu\} = \widehat{\partial} q(\bm{w}_\nu) = \left\{-\sum_{j=1}^m \bm{y}_j \cdot \mathbf{1}_{(\bm{w}_\nu - \bm{w})^\top \bm{y}_j>0} \right\} \subseteq G_{\textsf{PL}2}^L.
	\]
	As $G_{\textsf{PL}2}^L$ is a finite set, it is trivially closed with the usual Euclidean metric. We have $\partial q(\bm{w}) \subseteq G_{\textsf{PL}2}^L$. For the reverse direction, let $\bm{g}' \in G_{\textsf{PL}2}^L$. Then, there exists $\bm{d}$ such that
	\[
	\bm{g}' = -\sum_{j=1}^m \bm{y}_j \cdot \mathbf{1}_{\bm{d}^\top \bm{y}_j>0}
	\]
	with $\bm{d}^\top \bm{y}_j \neq \bm{0}$ for any $j \in [m]$.
	Let $\bm{w}_\nu = \bm{w} + \bm{d}/\nu$. We get $\bm{w}_\nu^\top \bm{y}_j = \nu^{-1}\bm{d}^\top \bm{y}_j \neq 0$ for any $j \in [m]$. Then, we know the function $q$ is strictly differentiable at $\bm{w}_\nu$ and $\{\bm{g}_\nu \} = \widehat\partial q(\bm{w}_\nu)$. Thus, for any $\nu$, we get $\bm{g}_\nu = \bm{g}'$. Consequently, we get $\bm{g}' \in \widehat\partial q(\bm{w})$ and $G_{\textsf{PL}2}^L \subseteq \partial q(\bm{w})$.
	
	(Necessary) Suppose $\bm{0}\neq \bm{v} \in \spn \big(\lB \bm{x}_i \rB_{i=1}^n \big) \cap 
	\spn \big(\lB \bm{y}_j \rB_{j=1}^m \big)$ and 
	\[
	\partial f_{\textsf{PL}}(\bm{w}) = \partial \left[h_{\textsf{PL}1}\left(\bm{U}_x^\top \cdot \right) \right](\bm{w}) + \partial  \left[-h_{\textsf{PL}2}\left(\bm{U}_y^\top \cdot \right) \right](\bm{w}) = G_{\textsf{PL}}^L.
	\] 
	Then, by taking a convex hull on both size and using \cite[Theorem 8.49]{rockafellar2009variational}, we get
	\[
	\partial_C f_{\textsf{PL}}(\bm{w}) = \partial_C \left[h_{\textsf{PL}1}\left(\bm{U}_x^\top \cdot \right) \right](\bm{w}) + \partial_C  \left[-h_{\textsf{PL}2}\left(\bm{U}_y^\top \cdot \right) \right](\bm{w}) = \conv (G_{\textsf{PL}}^L) = G_{\textsf{PL}}^C,
	\]
	which is a contradiction to \Cref{lem:s-nece}.
\end{proof}

\begin{proof}[Proof of \Cref{thm:reluMain-CR-limiting}]
	According to the argument in \Cref{sec:roadmap}, we only need to consider the limiting subdifferential $\partial f_k(\bm{w}_k)$ for every $k \in [H]$. It is showed in \Cref{thm:pl-CR-limiting} that we have
\[
\partial f_k(\bm{w}_k) = \widetilde{G}_k^L,
\]
if and only if the following span qualification is satisfied:
\[
\spn \left(\lB \bm{x}_i \rB_{i\in\mathcal{I}_k^+} \right) \cap 
	\spn \left(\lB \bm{x}_j \rB_{j\in\mathcal{I}_k^-} \right) = \{\bm{0}\}.
\]
Then, put all $k \in [H]$ cases together, and \Cref{thm:reluMain-CR-limiting} is proved.
\end{proof}

\subsubsection{Chain Rule for Fr\'echet Subdifferential}
\begin{Theorem}[Fr\'echet]\label{thm:pl-CR-frechet}
	Suppose $\bm{x}_i^\top \bm{w} = \bm{y}_j^\top \bm{w} = 0$ and $\bm{y}_j \neq \bm{0}$ for any $i\in [n], j \in [m]$. For any given $\bm{w}$ such that $\widehat{\partial} f_{\textsf{PL}} (\bm{w}) \neq \emptyset$, we have the following exact chain rule
		\[
	\widehat{\partial} f_{\textsf{PL}} (\bm{w}) = G_{\textsf{PL}}^F \coloneqq \sum_{i=1}^n \bm{x}_i\cdot [0,1] + \left\{ \begin{array}{rcl}
          \emptyset & \mbox{if}
         & m > 0, \\
        \{\bm{0}\} & \mbox{if} & m=0.
                \end{array}\right.
	\]
	if and only if
	$
	\spn \big(\lB \bm{x}_i \rB_{i=1}^n \big) \cap 
	\spn \big(\lB \bm{y}_j \rB_{j=1}^m \big) = \{\bm{0}\}.
	$
\end{Theorem}
\begin{proof}
(Sufficient) We begin with the general argument in the proof of \Cref{lem:s-suff} until \Cref{eq:sufficient-prep-end}. We will focus on the proof of 
	\[
	\widehat\partial  \left[-h_{\textsf{PL}2}\left(\bm{U}_y^\top \cdot \right) \right](\bm{w}) = G_{\textsf{PL}2}^F \coloneqq \left\{ \begin{array}{rcl}
          \emptyset & \mbox{if}
         & m > 0, \\
        \{\bm{0}\} & \mbox{if} & m = 0.
                \end{array}\right.
	\]
	For the ease of notation, we denote $q(\bm{w}) \coloneqq -h_{\textsf{PL}2}\left(\bm{U}_y^\top \bm{w} \right) = -\sum_{j=1}^m \max\{\bm{y}_j^\top\bm{w},0\}$. 
	Then, by \Cref{lem:frechet-empty}, we know that if there exists $j \in [m]$ such that $\bm{w}^\top \bm{y}_j = 0$ and $\bm{y}_j \neq \bm{0}$, then we have $\widehat{\partial} q(\bm{w}) = \emptyset$. If $m=0$, then $q(\bm{w}) = 0$ and $G_{\textsf{PL}2}^F=\{\bm{0}\}$. The claim follows trivially.
	
	(Necessary)  Suppose $\bm{0}\neq \bm{v} \in \spn \big(\lB \bm{x}_i \rB_{i=1}^n \big) \cap 
	\spn \big(\lB \bm{y}_j \rB_{j=1}^m \big)$. There exists $\bm{y}_j \neq \bm{0}$ as otherwise $\bm{v} \notin  \{\bm{0}\} \supseteq \spn \big(\lB \bm{y}_j \rB_{j=1}^m \big)$. Then, we get $m>0$ and $G_{\textsf{PL}}^F = \emptyset$. Thus, from the assumption that $\widehat{\partial} f_{\textsf{PL}} (\bm{w}) \neq \emptyset$, we know $\widehat{\partial} f_{\textsf{PL}} (\bm{w}) \supsetneq G_{\textsf{PL}}^F = \emptyset$ by definition.
\end{proof}

\begin{proof}[Proof of \Cref{thm:reluMain-CR-frechet}]
	According to the argument in \Cref{sec:roadmap}, we only need to consider the Fr\'echet subdifferential $\widehat\partial f_k(\bm{w}_k)$ for every $k \in [H]$. It is showed in \Cref{thm:pl-CR-frechet} that we have
\[
\widehat\partial f_k(\bm{w}_k) = G_k^F,
\]
if and only if the following span qualification is satisfied:
\[
\spn \left(\lB \bm{x}_i \rB_{i\in\mathcal{I}_k^+} \right) \cap 
	\spn \left(\lB \bm{x}_j \rB_{j\in\mathcal{I}_k^-} \right) = \{\bm{0}\}.
\]
Then, put all $k \in [H]$ cases together, and \Cref{thm:reluMain-CR-frechet} is proved.
\end{proof}

\subsection{Proofs for \Cref{sec:discussion}}\label{sec:apd-prf-discussion}

\begin{Definition}[Regularities]\label{def:detailed-regularity} We consider the following regularity conditions:
\begin{itemize}
	\item General position data \cite[Assumption 2]{yun2018efficiently}: No $d$ data points $\{\tilde{\bm{x}}_i\}_i \subseteq \mathbb{R}^{d-1}$ lie on the same affine hyperplane, which is equivalent to the nonexistence of $\bm{w} \in \mathbb{R}^d$ and index set $\mathcal{J}\subseteq [N]$ with $|\mathcal{J}| \geq d$ such that $\bm{w}^\top \bm{x}_j = 0$ for any $j\in\mathcal{J}$.
	\item Linear Independence Kink Qualification (LIKQ) \cite[Definition 2]{griewank2016first}, \cite[Definition 2.6]{griewank2019relaxing}: Let the $j$-th row of the matrix $\nabla \bm{z}^\sigma$ in \Cref{sec:abs-norm-form} be $\bm{v}_j^\top$. 
	We define the following index set
	\[
	\alpha\coloneqq \left\{N(k-1)+i: \bm{w}_k^\top \bm{x}_i = 0 \right\}
	=
	\{j : z_j = 0\}.
	\]
	LIKQ is satisfied if the vectors $\{\bm{v}_i\}_{i \in \alpha}$ are linearly independent.
	\item Linearly Independent Activated Data (LIAD): Let the index set $\mathcal{J}_k \coloneqq \{j:\bm{w}_k^\top\bm{x}_j=0\}$. For any fixed $k \in [H]$, the data points $\{\bm{x}_i\}_{i \in \mathcal{J}_k}$ are linearly independent.  
\end{itemize}
\end{Definition}

\begin{proof}[Proof of {\Cref{prop:regularity-relation}}]
	For the relation general position $\Longrightarrow$ LIAD, it directly follows from \cite[Lemma 1]{yun2018efficiently}. By the analysis in \Cref{sec:abs-relu}, we know LIKQ is satisfied for the empirical loss of  two-layer ReLU network if and only if 
	\[
\left\{\prod_{k'=1}^H 0 \times \bm{1}_{k'=k}\cdot \bm{x}_i \right\}_{(N(k-1)+i) \in \alpha}
\]
are linearly independent. 
 It is easy to see that LIKQ holds if and only if, for any given $k\in[H]$, the data points $\{\bm{x}_i\}_{i \in \{j:\bm{w}_k^\top\bm{x}_j=0\}}$ are linearly independent. Thus, we have the relation LIAD $\iff$ LIKQ. Note that $\mathcal{I}_k^+\cup \mathcal{I}_k^- = \{j:\bm{w}_k^\top\bm{x}_j=0\}$. If $\{\bm{x}_j\}_{j \in \mathcal{I}_k^+\cup \mathcal{I}_k^-}$ are linearly independent, then it is evident that
\[
\spn \left(\lB \bm{x}_i \rB_{i\in\mathcal{I}_k^+} \right) \cap 
	\spn \left(\lB \bm{x}_j \rB_{j\in\mathcal{I}_k^-} \right) = \{\bm{0}\},
\]
which implies LIAD $\Longrightarrow$ SQ.
\end{proof}

\begin{proof}[Proof of {\Cref{prop:lmin-regular}}]
	Under SQ, if the Fr\'echet subdifferential is nonempty, we get $\mathcal{I}_k^-=\emptyset$ for any $k \in [H]$. By \Cref{thm:reluMain-CR} and \Cref{thm:reluMain-CR-frechet}, we have $\partial_C L$ and $\widehat{\partial} L$ are equal at that point. Then, Clarke regularity follows from \Cref{def:clarke-regular}. By \Cref{prop:regularity-relation}, if the data points are in general position, then they satisfy SQ. Using \cite[Theorem 10.1]{rockafellar2009variational}, the Fr\'echet subdifferential is nonempty at every local minimizer, which completes the proof.
\end{proof}

\section{Proofs for \Cref{sec:test}}\label{sec:appd-test}

\subsection{Testing Clarke NAS}\label{sec:prf-thm-robust}

\begin{proof}[Proof of \Cref{thm:robust}]

We consider an $\epsilon$-Clarke stationary point $(u_1^*, \bm{w}_1^*, \dots, u_H^*,\bm{w}_H^*)$ with
\[
\left\| (u_1^*, \bm{w}_1^*, \dots, u_H^*,\bm{w}_H^*) \right\| \leq B.
\]
By \Cref{thm:reluMain-CR}, we know there exists $\bm{g}^* \in \partial_C L(u_1^*, \bm{w}_1^*, \dots, u_H^*,\bm{w}_H^*)$ such that
\[
\|\bm{g}^*\eqqcolon(g_1^*, \bm{g}_1^*, \dots, g_H^*, \bm{g}_H^*)\| = \dist\Big( \bm{0}, \partial_C L(u_1^*, \bm{w}_1^*, \dots, u_H^*,\bm{w}_H^*) \Big) \leq \epsilon.
\]
Note that $\widehat{u}_i = u_i$ for any $i \in [H]$ in the returned vector $(\widehat{u}_1, \widehat{\bm{w}}_1, \dots, \widehat{u}_H,\widehat{\bm{w}}_H)$ of \Cref{alg:rounding}. In this subsection, we will write $u_i$ rather than $\widehat{u}_i$ for simplicity. 
Given a positive radius $\delta \in (0, C_\tau^\textnormal{Clarke}]$, we aim to show that, for any
\[
(u_1, \bm{w}_1, \dots, u_H,\bm{w}_H) \in \mathbb{B}_{\delta}\Big( (u_1^*, \bm{w}_1^*, \dots, u_H^*,\bm{w}_H^*) \Big),
\]
we can certify that the rounded point returned by \Cref{alg:rounding} satisfies
\[
\dist\Big( \bm{0}, \partial_C L(u_1, \widehat{\bm{w}}_1, \dots, u_H,\widehat{\bm{w}}_H) \Big) \leq \epsilon + C_{\mu}^\textnormal{Clarke}\cdot\delta,
\]
where $C_\mu^\textnormal{Clarke}<+\infty$ is a constant depending on the curvature that we will discuss later. 

We define the following shorthands for convenience
\begin{alignat*}{2}
	\rho_i^* &\coloneqq \ell_i'\left( \sum_{k=1}^H u_k^*\cdot \max\left\{ (\bm{w}_k^*)^\top \bm{x}_i, 0 \right\} \right),\qquad &&\forall i \in[N], \\
	\widehat{\rho}_i &\coloneqq \ell_i'\left( \sum_{k=1}^H u_k\cdot \max\left\{ \widehat{\bm{w}}_k^\top \bm{x}_i, 0 \right\} \right),\qquad &&\forall i \in[N].
\end{alignat*}
Recall the definition of the rounded $\{\widehat{\bm{w}}_k\}_k$ and we define indices sets $\mathcal{J}_k^<,\mathcal{J}_k^=,\mathcal{J}_k^>$ as
\begin{alignat*}{2}
  \widehat{\bm{w}}_k = \argmin_{ \bm{z} \in \mathbb{R}^d} &\ \| \bm{z} - \bm{w}_k \|^2 \\
  \textnormal{s.t.}\ \  &\ \bm{z}^\top \bm{x}_i \geq 2R\cdot\delta, && \forall i \in \mathcal{J}_k^>\coloneqq  \left\{j\in[N]: \bm{x}_j^\top \bm{w}_k > R\cdot \delta \right\}, \\
  &\ \bm{z}^\top \bm{x}_i \leq -2R\cdot\delta, \qquad&& \forall i \in \mathcal{J}_k^< \coloneqq \left\{j\in [N]: \bm{x}_j^\top \bm{w}_k < -R\cdot \delta\right\}, \\
  &\ \bm{z}^\top \bm{x}_i =0, && \forall i \in \mathcal{J}_k^=\coloneqq  \left\{j\in[N]: \left|\bm{x}_j^\top \bm{w}_k\right| \leq R\cdot \delta \right\}.
\end{alignat*}
We consider the following quantity related to the point $\bm{w}_k^*$ for any $k\in[H]$:
\[
\tau_k \coloneqq \min\left\{ \quad \min_{i :\bm{x}_i^\top\bm{w}_k^* > 0} \bm{x}_i^\top\bm{w}_k^*, \quad -\max_{i :\bm{x}_i^\top\bm{w}_k^* < 0} \bm{x}_i^\top\bm{w}_k^*\quad  \right\}.
\]
Note that $0 < \delta \leq C_\tau^{\textnormal{Clarke}} \leq  \frac{\tau_k}{4R}$.
For any $i\in[N]$ such that $\bm{x}_i^\top\bm{w}_k^*>0$, we have
\begin{align*}
\bm{x}_i^\top\bm{w}_k &= \bm{x}_i^\top\bm{w}_k^* - \bm{x}_i^\top\left(\bm{w}_k^* - \bm{w}_k\right) \\
&\geq \bm{x}_i^\top\bm{w}_k^* - \|\bm{x}_i\|\cdot \|\bm{w}_k^* - \bm{w}_k\| \\
&\geq \tau_k - R\cdot \delta  \geq 3R\cdot \delta > R\cdot \delta.
\end{align*}
Thus, we know $\left\{i:\bm{x}_i^\top\bm{w}_k^*>0\right\} \subseteq \mathcal{J}_k^>$. Similarly, for any $i\in[N]$ such that $\bm{x}_i^\top\bm{w}_k^*<0$, we have
\begin{align*}
\bm{x}_i^\top\bm{w}_k &= \bm{x}_i^\top\bm{w}_k^* + \bm{x}_i^\top\left(\bm{w}_k - \bm{w}_k^*\right) \\
&\leq \bm{x}_i^\top\bm{w}_k^* + \|\bm{x}_i\|\cdot \|\bm{w}_k^* - \bm{w}_k\| \\
&\leq -\tau_k + R\cdot \delta  \leq -3R\cdot \delta < -R\cdot \delta,
\end{align*}
which implies $\left\{i:\bm{x}_i^\top\bm{w}_k^*<0\right\} \subseteq \mathcal{J}_k^<$.
We have, for any $i\in[N]$ such that $\bm{x}_i^\top\bm{w}_k^*=0$, it holds
\begin{align*}
\left|\bm{x}_i^\top\bm{w}_k \right| &\leq \left| \bm{x}_i^\top\bm{w}_k^* \right| + \left|\bm{x}_i^\top\left(\bm{w}_k - \bm{w}_k^*\right) \right|\\
&\leq \left|\bm{x}_i^\top\bm{w}_k^*\right| + \|\bm{x}_i\|\cdot \|\bm{w}_k^* - \bm{w}_k\| \\
&\leq R\cdot \delta.
\end{align*}
So, we know $\left\{i:\bm{x}_i^\top\bm{w}_k^*=0\right\} \subseteq \mathcal{J}_k^=$. As 
$\mathcal{J}_k^<,\mathcal{J}_k^=,\mathcal{J}_k^>$ are disjoint and $[N]=\mathcal{J}_k^< \sqcup\mathcal{J}_k^= \sqcup \mathcal{J}_k^>$, we know 
\[
\left\{i:\bm{x}_i^\top\bm{w}_k^*<0\right\}=\mathcal{J}_k^<, \quad \left\{i:\bm{x}_i^\top\bm{w}_k^*=0\right\}=\mathcal{J}_k^=, \quad \left\{i:\bm{x}_i^\top\bm{w}_k^*>0\right\}=\mathcal{J}_k^>.
\]
Meanwhile, as $\widehat{\bm{w}}_k$ is feasible to the quadratic program in \Cref{alg:rounding}, we get
\[
\left\{i:\bm{x}_i^\top\widehat{\bm{w}}_k<0\right\}=\mathcal{J}_k^<, \quad \left\{i:\bm{x}_i^\top\widehat{\bm{w}}_k=0\right\}=\mathcal{J}_k^=, \quad \left\{i:\bm{x}_i^\top\widehat{\bm{w}}_k>0\right\}=\mathcal{J}_k^>,
\]
which implies $\mathcal{I}_k^+(\bm{w}_k^*)\cup\mathcal{I}_k^-(\bm{w}_k^*) = \mathcal{J}_k^= = \mathcal{I}_k^+(\widehat{\bm{w}}_k)\cup\mathcal{I}_k^-(\widehat{\bm{w}}_k)$ and $\mathbf{1}_{\bm{x}_i ^\top\bm{w}_k^*> 0} = \mathbf{1}_{i \in \mathcal{J}_k^> } = \mathbf{1}_{\bm{x}_i ^\top\widehat{\bm{w}}_k> 0}$ for any $k \in [H]$.
It is evident that 
\[
\|(\widehat{\bm{w}}_1,\dots,\widehat{\bm{w}}_H)  - (\bm{w}_1, \dots, \bm{w}_H) \| \leq \|(\bm{w}_1, \dots, \bm{w}_H) - (\bm{w}_1^*, \dots, \bm{w}_H^*) \| \leq \delta,
\]
as, for any $k \in [H]$, $\bm{w}_k^*$ is feasible to the quadratic program for computing $\widehat{\bm{w}}_k$ in \Cref{alg:rounding}. Therefore, we know 
\[
(u_1, \widehat{\bm{w}}_1, \dots, u_H,\widehat{\bm{w}}_H) \in \mathbb{B}_{\delta}\Big( (u_1, \bm{w}_1, \dots, u_H,\bm{w}_H) \Big).
\] 
By triangle inequality, it holds that
\[
(u_1, \widehat{\bm{w}}_1, \dots, u_H,\widehat{\bm{w}}_H) \in \mathbb{B}_{2\delta}\Big( (u_1^*, \bm{w}_1^*, \dots, u_H^*,\bm{w}_H^*) \Big).
\]
Using \Cref{thm:reluMain-CR}, we get
\[
\begin{aligned}
	&\dist\Big( \bm{0}, \partial_C L(u_1, \widehat{\bm{w}}_1, \dots, u_H,\widehat{\bm{w}}_H) \Big) \\
	&\leq \|\bm{g}^*\| + \dist\Big( \bm{g}^*, \partial_C L(u_1, \widehat{\bm{w}}_1, \dots, u_H,\widehat{\bm{w}}_H) \Big) \\
	&\leq \epsilon + \sum_{k=1}^H  \left| g_k^*- \sum_{i=1}^N \widehat{\rho}_i \cdot \max\left\{ \widehat{\bm{w}}_k^\top \bm{x}_i, 0 \right\}  \right| + \sum_{k=1}^H\dist\Big( \bm{g}^*_k, \partial_C \overline{L}_k(\widehat{\bm{w}}_k) \Big),
\end{aligned}
\]
where we define $\overline{L}_k(\widehat{\bm{w}}_k) = \sum_{i=1}^N u_k\widehat{\rho}_i\cdot \max\{\bm{x}_i^\top \widehat{\bm{w}}_k,0\}$.
We first compute
\begin{align*}
\sum_{i=1}^N \Big|\widehat{\rho}_i-\rho_i^*\Big|&= \sum_{i=1}^N \left|\ell_i'\left( \sum_{k=1}^H u_k\cdot \max\left\{ \widehat{\bm{w}}_k^\top \bm{x}_i, 0 \right\} \right) - \ell_i'\left( \sum_{k=1}^H u_k^*\cdot \max\left\{ (\bm{w}^*_k)^\top \bm{x}_i, 0 \right\} \right) \right| \\
&\leq L_{\ell'}\cdot \sum_{i=1}^N \left \| \sum_{k=1}^H  \left( u_k\cdot \max\left\{ \widehat{\bm{w}}_k^\top \bm{x}_i, 0 \right\} - u_k^*\cdot \max\left\{ (\bm{w}^*_k)^\top \bm{x}_i, 0 \right\}  \right)\right\| \\
&\leq L_{\ell'}\cdot \sum_{i=1}^N  \sum_{k=1}^H \big( |u_k|\cdot \|\bm{x}_i\| \cdot \| \widehat{\bm{w}}_k - \bm{w}^*_k \| + \|\bm{w}_k^*\|\cdot \|\bm{x}_i\|\cdot |u_k - u_k^*| \big) \\
&\leq 3L_{\ell'}NHBR\cdot \delta \eqqcolon C_1  \cdot  \delta.
\end{align*}
We now upper bound the second term $\sum_{k=1}^H  \left| g_k^*- \sum_{i=1}^N \widehat{\rho}_i \cdot \max\left\{ \widehat{\bm{w}}_k^\top \bm{x}_i, 0 \right\}  \right|$. Note that

\begin{align*}
	\left| g_k^* -  \sum_{i=1}^N \widehat{\rho}_i \cdot \max\left\{ \widehat{\bm{w}}_k^\top \bm{x}_i, 0 \right\}  \right| 
	&= \left| \sum_{i=1}^N \rho_i^* \cdot \max\left\{ (\bm{w}_k^*)^\top \bm{x}_i, 0 \right\}  -   \sum_{i=1}^N \widehat{\rho}_i \cdot \max\left\{ \widehat{\bm{w}}_k^\top \bm{x}_i, 0 \right\}  \right| \\
	&\leq  \underbrace{ \left|\sum_{i=1}^N \rho_i^* \cdot \max\left\{ (\bm{w}_k^*)^\top \bm{x}_i, 0 \right\}  -  \sum_{i=1}^N \widehat{\rho}_i \cdot \max\left\{ (\bm{w}_k^*)^\top \bm{x}_i, 0 \right\}  \right| }_{\eqqcolon T_1^k} \\
	&\qquad+ \underbrace{ \left|  \sum_{i=1}^N \widehat{\rho}_i \cdot \max\left\{ (\bm{w}_k^*)^\top \bm{x}_i, 0 \right\}  -   \sum_{i=1}^N \widehat{\rho}_i \cdot \max\left\{ \widehat{\bm{w}}_k^\top \bm{x}_i, 0 \right\}  \right| }_{\eqqcolon T_2^k}.
\end{align*}
Now we estimate these two terms.
For $T_1^k$, we compute
\begin{align*}
	T_1^k &\leq  \sum_{i=1}^N \|\bm{w}_k^*\|\cdot \|\bm{x}_i\| \cdot \left| \widehat{\rho}_i - \rho_i^* \right| 
	\leq BR \cdot \sum_{i=1}^N \left| \widehat{\rho}_i - \rho_i^* \right|  \leq BR \cdot C_1\cdot\delta \eqqcolon C_2 \cdot \delta.
\end{align*}
For $T_2^k$, we see that
\begin{align*}
	T_2^k &\leq  \sum_{i=1}^N \left| \widehat{\rho}_i\right| \cdot \|\bm{x}_i\|\cdot \|\bm{w}_k - \bm{w}_k^*\| \leq 2L_{\ell'}NR\cdot \delta \eqqcolon C_3 \cdot \delta.
\end{align*}
Summarizing, we have
\[
\sum_{k=1}^H  \left| g_k^*- \sum_{i=1}^N \widehat{\rho}_i \cdot \max\left\{ \widehat{\bm{w}}_k^\top \bm{x}_i, 0 \right\}  \right| \leq \sum_{k=1}^H T_1^k + T_2^k \leq H(C_2+C_3)\cdot \delta \eqqcolon C_4\cdot \delta.
\]
We proceed to upper bound $\sum_{k=1}^H\dist\Big( \bm{g}^*_k, \partial_C \overline{L}_k(\widehat{\bm{w}}_k) \Big)$.

By \Cref{thm:reluMain-CR}, we know that there exist $\xi_j \in [0,1], \forall j \in \mathcal{I}_k^+(\bm{w}_k^*)\cup\mathcal{I}_k^-(\bm{w}_k^*)$ such that the Clarke subgradient $\bm{g}_k^* \in \partial_C \overline{L}_k(\bm{w}_k^*)$ can be written as
\[
\bm{g}_k^*\coloneqq \sum_{i \in [N]\backslash \big(\mathcal{I}_k^+(\bm{w}_k^*)\cup\mathcal{I}_k^-(\bm{w}_k^*)\big)} u_k^*\rho_i^*\cdot \mathbf{1}_{\bm{x}_i ^\top\bm{w}_k^*> 0} \cdot\bm{x}_i + \sum_{j\in  \mathcal{I}_k^+(\bm{w}_k^*)\cup\mathcal{I}_k^-(\bm{w}_k^*)} u_k^*\rho_j^*\cdot \bm{x}_j\cdot \xi_j.
\]
Now, we are well prepared to upper bound $\dist\big(\bm{g}_k^*, \partial_C \overline L_k(\widehat{\bm{w}}_k)\big)$. Let
\[
\widehat{\bm{g}}_k \coloneqq \sum_{i \in [N]\backslash \big(\mathcal{I}_k^+(\widehat{\bm{w}}_k)\cup\mathcal{I}_k^-(\widehat{\bm{w}}_k)\big)} u_k\widehat{\rho}_i\cdot \mathbf{1}_{\bm{x}_i ^\top\widehat{\bm{w}}_k> 0} \cdot\bm{x}_i
+
\sum_{j\in \mathcal{I}_k^+(\widehat{\bm{w}}_k)\cup\mathcal{I}_k^-(\widehat{\bm{w}}_k)} u_k\widehat{\rho}_j\cdot \bm{x}_j\cdot \xi_j,
\]
which, by \Cref{thm:reluMain-CR}, belongs to the Clarke subdifferential $\partial_C \overline L_k(\widehat{\bm{w}}_k)$. We upper bound
\[
\dist\big(\bm{g}_k^*, \partial_C \overline L_k(\widehat{\bm{w}}_k)\big) \leq  \|\widehat{\bm{g}}_k - \bm{g}_k^*\|
\]
with
\begin{align*}
&\|\widehat{\bm{g}}_k - \bm{g}_k^*\| \\
&= \left\| 
\sum_{i \in [N]\backslash \big(\mathcal{I}_k^+(\widehat{\bm{w}}_k)\cup\mathcal{I}_k^-(\widehat{\bm{w}}_k)\big)}  \Big(u_k\widehat{\rho}_i-u_k^*\rho_i^*\Big)\cdot \mathbf{1}_{\bm{x}_i ^\top\widehat{\bm{w}}_k> 0} \cdot\bm{x}_i
+
\sum_{j\in \mathcal{I}_k^+(\widehat{\bm{w}}_k)\cup\mathcal{I}_k^-(\widehat{\bm{w}}_k)} \Big(u_k\widehat{\rho}_j-u_k^*\rho_j^*\Big)\cdot \bm{x}_j\cdot \xi_j
\right\| \\
&\leq 
\sum_{1\leq i \leq N}  \|\bm{x}_i\| \cdot \Big|u_k\widehat{\rho}_i-u_k^*\rho_i^*\Big| \\
&\leq R \cdot \sum_{1\leq i \leq N} \left(|u_k|\cdot\Big|\widehat{\rho}_i-\rho_i^*\Big| + |\rho_i^*|\cdot \Big|u_k-u_k^*\Big| \right) \\
&\leq  BR\cdot C_1\cdot \delta + NRL_{\ell'}\cdot \delta.
\end{align*}
Then, we have
\[
\sum_{k=1}^H \dist\big(\bm{g}_k^*, \partial_C \overline L_k(\widehat{\bm{w}}_k)\big) \leq  \sum_{k=1}^H  \|\widehat{\bm{g}}_k - \bm{g}_k^*\| \leq H(BR\cdot C_1 + NRL_{\ell'})\cdot \delta \eqqcolon  C_5 \cdot\delta.
\]
In sum, we have proved that
\[
\dist\Big( \bm{0}, \partial_C L(u_1, \widehat{\bm{w}}_1, \dots, u_H,\widehat{\bm{w}}_H) \Big) \leq \epsilon + C_\mu^{\textnormal{Clarke}}\cdot\delta,
\]
where $C_\mu^{\textnormal{Clarke}}\coloneqq C_4 + C_5=\textnormal{poly}(B,R,L_\ell, L_{\ell'}, N, H)$.
\end{proof}

\subsection{Testing Fr\'echet NAS}\label{sec:appd-test-f}

\begin{proof}[Proof of \Cref{thm:robust-f}]
Some steps in the computation are similar to these in the proof of \Cref{thm:robust} in \Cref{sec:prf-thm-robust}, and we may skip them for simplicity.
We consider an $\epsilon$-Fr\'echet stationary point $(u_1^*, \bm{w}_1^*, \dots, u_H^*,\bm{w}_H^*)$ with
$
\left\| (u_1^*, \bm{w}_1^*, \dots, u_H^*,\bm{w}_H^*) \right\| \leq B.
$
By \Cref{thm:reluMain-CR-frechet}, there exists a regular subgradient $\bm{g}^* \in \widehat{\partial} L(u_1^*, \bm{w}_1^*, \dots, u_H^*,\bm{w}_H^*)$ such that
\[
\|\bm{g}^*\eqqcolon(g_1^*, \bm{g}_1^*, \dots, g_H^*, \bm{g}_H^*)\| = \dist\Big( \bm{0}, \widehat{\partial} L(u_1^*, \bm{w}_1^*, \dots, u_H^*,\bm{w}_H^*) \Big) \leq \epsilon.
\]
Given a positive radius $\delta \in (0, C_\tau^{\textnormal{Fr\'echet}}]$, we aim to show that, for any
\[
(u_1, \bm{w}_1, \dots, u_H,\bm{w}_H) \in \mathbb{B}_{\delta}\Big( (u_1^*, \bm{w}_1^*, \dots, u_H^*,\bm{w}_H^*) \Big),
\]
we can certify the rounded point returned by \Cref{alg:rounding-f} satisfying 
\[
\dist\Big( \bm{0}, \widehat{\partial} L(\widehat{u}_1, \widehat{\bm{w}}_1, \dots, \widehat{u}_H,\widehat{\bm{w}}_H) \Big) \leq \epsilon + C_{\mu}^{\textnormal{Fr\'echet}}\cdot\delta,
\]
where $C_\mu^{\textnormal{Fr\'echet}}<+\infty$ is a constant depending on the curvature that we will discuss later. 

Similar to \Cref{sec:prf-thm-robust}, we define the following shorthands for convenience
\begin{alignat*}{2}
	\rho_i^* &\coloneqq \ell_i'\left( \sum_{k=1}^H u_k^*\cdot \max\left\{ (\bm{w}_k^*)^\top \bm{x}_i, 0 \right\} \right),\qquad &&\forall i \in[N], \\
	\widehat{\rho}_i &\coloneqq \ell_i'\left( \sum_{k=1}^H \widehat{u}_k\cdot \max\left\{ \widehat{\bm{w}}_k^\top \bm{x}_i, 0 \right\} \right),\qquad &&\forall i \in[N].
\end{alignat*}
We consider the following quantity related to the point $\bm{w}_k^*$ for any $k\in[H]$:
\[
\tau_k \coloneqq \min\left\{ \quad \min_{i :\bm{x}_i^\top\bm{w}_k^* > 0} \bm{x}_i^\top\bm{w}_k^*, \quad -\max_{i :\bm{x}_i^\top\bm{w}_k^* < 0} \bm{x}_i^\top\bm{w}_k^*\quad  \right\}.
\]
We use the same indices sets $\mathcal{J}_k^<,\mathcal{J}_k^=,\mathcal{J}_k^>$ for computing the rounded $\{\widehat{\bm{w}}_k\}_k$ as those in \Cref{sec:prf-thm-robust}. Note that $0 < \delta \leq C_\tau^{\textnormal{Fr\'echet}} \leq  \frac{\tau_k}{4R}$. The argument in \Cref{sec:prf-thm-robust} shows that 
\[
\left\{i:\bm{x}_i^\top\bm{w}_k^*<0\right\}=\mathcal{J}_k^<, \quad \left\{i:\bm{x}_i^\top\bm{w}_k^*=0\right\}=\mathcal{J}_k^=, \quad \left\{i:\bm{x}_i^\top\bm{w}_k^*>0\right\}=\mathcal{J}_k^>.
\]
Meanwhile, as $\widehat{\bm{w}}_k$ is feasible to the quadratic program in \Cref{alg:rounding-f}, we get
\[
\left\{i:\bm{x}_i^\top\widehat{\bm{w}}_k<0\right\}=\mathcal{J}_k^<, \quad \left\{i:\bm{x}_i^\top\widehat{\bm{w}}_k=0\right\}=\mathcal{J}_k^=, \quad \left\{i:\bm{x}_i^\top\widehat{\bm{w}}_k>0\right\}=\mathcal{J}_k^>,
\]
which implies $\mathcal{I}_k^+(\bm{w}_k^*)\cup\mathcal{I}_k^-(\bm{w}_k^*) = \mathcal{J}_k^= = \mathcal{I}_k^+(\widehat{\bm{w}}_k)\cup\mathcal{I}_k^-(\widehat{\bm{w}}_k)$ and $\mathbf{1}_{\bm{x}_i ^\top\bm{w}_k^*> 0} = \mathbf{1}_{i \in \mathcal{J}_k^> } = \mathbf{1}_{\bm{x}_i ^\top\widehat{\bm{w}}_k> 0}$ for any $k \in [H]$.
It is evident that 
\[
\|(\widehat{\bm{w}}_1,\dots,\widehat{\bm{w}}_H)  - (\bm{w}_1, \dots, \bm{w}_H) \| \leq \|(\bm{w}_1, \dots, \bm{w}_H) - (\bm{w}_1^*, \dots, \bm{w}_H^*) \|,
\]
as, for any $k \in [H]$, $\bm{w}_k^*$ is feasible to the quadratic program for computing $\widehat{\bm{w}}_k$ in \Cref{alg:rounding-f}.

However, the identification of $\bm{w}_k^*$ is not sufficient to bound $\dist\Big( \bm{g}^*_k, \widehat{\partial} \overline{L}_k(\widehat{\bm{w}}_k) \Big)$, as the index set $\mathcal{I}_k^-(\widehat{\bm{w}}_k)$ may not be an empty set, which, by \Cref{thm:reluMain-CR-frechet}, implies $G_k^F = \emptyset$ and $\dist\Big( \bm{g}^*_k, \widehat{\partial} \overline{L}_k(\widehat{\bm{w}}_k) \Big)=+\infty$.
We are thus looking for the identification of $\{u_k^*\}_{k=1}^H$. We define a constant $C_u\coloneqq L_{\ell'}(4HRB^2+1)$ and consider the following quantity related to the point $u_k^*$ for any $k\in[H]$:
\[
\tau_k' \coloneqq \min \Big\{ u_k^*\cdot \rho_i^*: i \in \mathcal{J}_k^=, u_k^*\cdot \rho_i^* > 0 \Big\}.
\]

Note that $0 < \delta \leq C_\tau^{\textnormal{Fr\'echet}} \leq \frac{\tau_k'}{4C_u}$. Fix any $k \in [H]$ and we consider two cases. If $u_k^* \neq 0$, by \Cref{thm:reluMain-CR-frechet} and \Cref{assu:frechet-nondegenerate}, we know $u_k^*\cdot \rho_i^* > 0$ for any $i \in \mathcal{J}_k^=$. Then, for any $i \in \mathcal{J}_k^=$, we have
\[
\begin{aligned}
	u_k\cdot \rho_i &= u_k^*\cdot \rho_i^* + (u_k\cdot \rho_i - u_k^*\cdot \rho_i^*) \geq u_k^*\cdot \rho_i^* - \left|u_k\cdot \rho_i - u_k^*\cdot \rho_i^*\right| \\
	&\geq \tau_k' - |u_k|\cdot |\rho_i - \rho_i^*| -  |\rho_k^*|\cdot |u_k - u_k^*| \\
	&\geq \tau_k' - C_u\cdot \delta  \geq 3C_u\cdot\delta,
\end{aligned}
\]
which by the rounding step of $\widehat{u}_k$ in \Cref{alg:rounding-f} implies if $u_k^* \neq 0$, then $\widehat{u}_k = {u}_k$  and 
\[
\widehat{u}_k\cdot \widehat{\rho}_i \geq  u_k\cdot \rho_i - | u_k|\cdot |\widehat{\rho}_i - \rho_i| \geq 3C_u\cdot \delta - C_u\cdot \delta >0.
\]
If $u_k^* = 0$, we can see that
\[
|u_k\cdot \rho_i| = |u_k\cdot \rho_i - u_k^*\cdot \rho_k^*| \leq C_u\delta,
\]
which implies $\widehat{u}_k = 0$ by rounding step of $\widehat{u}_k$ in \Cref{alg:rounding-f}. Thus, we have proved that for any $k\in[H]$ and $i \in \mathcal{J}_k^=$, we get $\widehat{u}_k\cdot \widehat{\rho}_i \geq 0$, hence that $\mathcal{I}_k^-(\widehat{\bm{w}}_k) = \mathcal{I}_k^-(\bm{w}_k^*) = \emptyset$, and finally that $\mathcal{I}_k^+(\widehat{\bm{w}}_k)=\mathcal{I}_k^+(\bm{w}_k^*)$. By \Cref{thm:reluMain-CR-frechet}, we conclude that $\widehat{\partial}L(\widehat{u}_1, \widehat{\bm{w}}_1, \dots, \widehat{u}_H, \widehat{\bm{u}}_H) \neq \emptyset$.

Summarizing, we have 
\begin{align*}
	&\Big\|(\widehat{u}_1, \widehat{\bm{w}}_1,\dots,\widehat{u}_H,\widehat{\bm{w}}_H)  - (u_1,\bm{w}_1, \dots, u_H,\bm{w}_H) \Big\|^2 \\ 
	&=  \Big\|(\widehat{\bm{w}}_1, \dots, \widehat{\bm{w}}_H) - (\bm{w}_1, \dots, \bm{w}_H) \Big\|^2+ \sum_{k:u_k^* = 0} | \widehat{u}_k - u_k |^2 \tag{${u}_k=\widehat{u}_k$ if $u_k^* \neq 0$} \\
	& \leq \Big\|(\bm{w}_1, \dots, \bm{w}_H) - (\bm{w}_1^*, \dots, \bm{w}_H^*) \Big\|^2+ \sum_{k:u_k^* = 0} | {u}_k^* - u_k |^2, \tag{by ${u}_k^*=\widehat{u}_k=0$} \\
	&\leq \Big\| (u_1,\bm{w}_1, \dots, u_H,\bm{w}_H) -  (u_1^*,\bm{w}_1^*, \dots, u_H^*,\bm{w}_H^*) \Big\|^2 \leq \delta^2.
\end{align*}
This shows by triangle inequality  that 
\[
\mathbb{B}_{2\delta}\Big( (u_1^*, \bm{w}_1^*, \dots, u_H^*,\bm{w}_H^*) \Big) \ni
(\widehat{u}_1, \widehat{\bm{w}}_1, \dots, \widehat{u}_H,\widehat{\bm{w}}_H) \in \mathbb{B}_{\delta}\Big( (u_1, \bm{w}_1, \dots, u_H,\bm{w}_H) \Big).
\] 
Using \Cref{thm:reluMain-CR-frechet}, we get
\[
\begin{aligned}
	&\dist\Big( \bm{0}, \widehat{\partial} L(\widehat{u}_1, \widehat{\bm{w}}_1, \dots, \widehat{u}_H,\widehat{\bm{w}}_H) \Big) \leq \|\bm{g}^*\| + \dist\Big( \bm{g}^*, \widehat{\partial} L(\widehat{u}_1, \widehat{\bm{w}}_1, \dots, \widehat{u}_H,\widehat{\bm{w}}_H) \Big) \\
	&\leq \epsilon + \sum_{k=1}^H  \left| g_k^*- \sum_{i=1}^N \widehat{\rho}_i \cdot \max\left\{ \widehat{\bm{w}}_k^\top \bm{x}_i, 0 \right\}  \right| + \sum_{k=1}^H\dist\Big( \bm{g}^*_k, \widehat{\partial} \overline{L}_k(\widehat{\bm{w}}_k) \Big),
\end{aligned}
\]
where we define $\overline{L}_k(\widehat{\bm{w}}_k) = \sum_{i=1}^N \widehat{u}_k\widehat{\rho}_i\cdot \max\{\bm{x}_i^\top \widehat{\bm{w}}_k,0\}$.
We first compute
\begin{align*}
\sum_{i=1}^N \Big|\widehat{\rho}_i-\rho_i^*\Big|&= \sum_{i=1}^N \left|\ell_i'\left( \sum_{k=1}^H \widehat{u}_k\cdot \max\left\{ \widehat{\bm{w}}_k^\top \bm{x}_i, 0 \right\} \right) - \ell_i'\left( \sum_{k=1}^H u_k^*\cdot \max\left\{ (\bm{w}^*_k)^\top \bm{x}_i, 0 \right\} \right) \right| \\
&\leq L_{\ell'}\cdot \sum_{i=1}^N  \sum_{k=1}^H \big( |\widehat{u}_k|\cdot \|\bm{x}_i\| \cdot \| \widehat{\bm{w}}_k - \bm{w}^*_k \| + \|\bm{w}_k^*\|\cdot \|\bm{x}_i\|\cdot |\widehat{u}_k - u_k^*| \big) \\
&\leq 4L_{\ell'}NHBR\cdot \delta \eqqcolon C_1  \cdot  \delta.
\end{align*}
We now upper bound the second term $\sum_{k=1}^H  \left| g_k^*- \sum_{i=1}^N \widehat{\rho}_i \cdot \max\left\{ \widehat{\bm{w}}_k^\top \bm{x}_i, 0 \right\}  \right|$. A computation similar to that in \Cref{sec:prf-thm-robust} shows that
	\[
\sum_{k=1}^H  \left| g_k^*- \sum_{i=1}^N \widehat{\rho}_i \cdot \max\left\{ \widehat{\bm{w}}_k^\top \bm{x}_i, 0 \right\}  \right| \leq H(C_2+C_3)\cdot \delta \eqqcolon C_4\cdot \delta.
\]
where $C_2\coloneqq BR \cdot C_1$ and $C_3\coloneqq 2L_{\ell'}NR$.

We proceed to upper bound $\sum_{k=1}^H\dist\Big( \bm{g}^*_k, \widehat{\partial} \overline{L}_k(\widehat{\bm{w}}_k) \Big)$.
By \Cref{thm:reluMain-CR-frechet}, we know that there exist $\xi_j \in [0,1], \forall j \in \mathcal{I}_k^+(\bm{w}_k^*)$ such that $\bm{g}_k^* \in \widehat{\partial} \overline{L}_k(\bm{w}_k^*)$ can be written as
\[
\bm{g}_k^*\coloneqq \sum_{i \in [N]\backslash \big(\mathcal{I}_k^+(\bm{w}_k^*)\cup\mathcal{I}_k^-(\bm{w}_k^*)\big)} u_k^*\rho_i^*\cdot \mathbf{1}_{\bm{x}_i ^\top\bm{w}_k^*> 0} \cdot\bm{x}_i + \sum_{j\in  \mathcal{I}_k^+(\bm{w}_k^*)} u_k^*\rho_j^*\cdot \bm{x}_j\cdot \xi_j.
\]
Now, we are well prepared to upper bound $\dist\big(\bm{g}_k^*, \widehat{\partial} \overline L_k(\widehat{\bm{w}}_k)\big)$. Let
\[
\widehat{\bm{g}}_k \coloneqq \sum_{i \in [N]\backslash \big(\mathcal{I}_k^+(\widehat{\bm{w}}_k)\cup\mathcal{I}_k^-(\widehat{\bm{w}}_k)\big)} \widehat{u}_k\widehat{\rho}_i\cdot \mathbf{1}_{\bm{x}_i ^\top\widehat{\bm{w}}_k> 0} \cdot\bm{x}_i
+
\sum_{j\in \mathcal{I}_k^+(\widehat{\bm{w}}_k)} \widehat{u}_k\widehat{\rho}_j\cdot \bm{x}_j\cdot \xi_j,
\]
which, by \Cref{thm:reluMain-CR-frechet}, belongs to the Fr\'echet subdifferential $\widehat{\partial} \overline L_k(\widehat{\bm{w}}_k)$. We proceed to upper bound
$
\dist\big(\bm{g}_k^*, \widehat{\partial} \overline L_k(\widehat{\bm{w}}_k)\big) \leq  \|\widehat{\bm{g}}_k - \bm{g}_k^*\|
$
with
\begin{align*}
&\|\widehat{\bm{g}}_k - \bm{g}_k^*\| \\
&= \left\| 
\sum_{i \in [N]\backslash \big(\mathcal{I}_k^+(\widehat{\bm{w}}_k)\cup\mathcal{I}_k^-(\widehat{\bm{w}}_k)\big)}  \Big(\widehat{u}_k\widehat{\rho}_i-u_k^*\rho_i^*\Big)\cdot \mathbf{1}_{\bm{x}_i ^\top\widehat{\bm{w}}_k> 0} \cdot\bm{x}_i
+
\sum_{j\in \mathcal{I}_k^+(\widehat{\bm{w}}_k)} \Big(\widehat{u}_k\widehat{\rho}_j-u_k^*\rho_j^*\Big)\cdot \bm{x}_j\cdot \xi_j
\right\| \\
&\leq R \cdot \sum_{1\leq i \leq N} \left(|\widehat{u}_k|\cdot\Big|\widehat{\rho}_i-\rho_i^*\Big| + |\rho_i^*|\cdot \Big|\widehat{u}_k-u_k^*\Big| \right) \\
&\leq  BR\cdot C_1\cdot \delta + 2NRL_{\ell'}\cdot \delta.
\end{align*}
Then, we have
\[
\sum_{k=1}^H \dist\big(\bm{g}_k^*, \widehat{\partial} \overline L_k(\widehat{\bm{w}}_k)\big) \leq  \sum_{k=1}^H  \|\widehat{\bm{g}}_k - \bm{g}_k^*\| \leq H(BR\cdot C_1 + 2NRL_{\ell'})\cdot \delta \eqqcolon  C_5 \cdot\delta.
\]
In sum, we have proved that
\[
\dist\Big( \bm{0}, \widehat{\partial} L(\widehat{u}_1, \widehat{\bm{w}}_1, \dots, \widehat{u}_H,\widehat{\bm{w}}_H) \Big) \leq \epsilon + C_\mu^{\textnormal{Fr\'echet}}\cdot\delta,
\]
where $C_\mu^{\textnormal{Fr\'echet}}\coloneqq C_4 + C_5 = \textnormal{poly}(B,R,L_\ell, L_{\ell'}, N, H)$.
\end{proof}

\bibliography{ref}
\bibliographystyle{abbrvnat}

\newpage

\end{document}